\pdfoutput=1
\documentclass[10pt,a4paper]{article}

\usepackage[utf8]{inputenc}
\usepackage[english]{babel}
\usepackage{amsmath}
\usepackage{amsfonts}
\usepackage{amssymb}
\usepackage{mathrsfs}
\usepackage{amsthm}
\usepackage{latexsym,amsmath,amstext,amssymb,mathtools}
\usepackage{caption}
\usepackage{commath} 
\usepackage{hhline}
\usepackage{listings}
\usepackage{enumerate}
\usepackage{color}
\usepackage{xfrac}
\usepackage{faktor}
\usepackage{colonequals}

\usepackage{xr} 
\usepackage[toc,page]{appendix} 
\usepackage{hyphenat}
\usepackage{stmaryrd} 
\usepackage{csquotes}
\usepackage{array}
\usepackage{indentfirst} 

\usepackage{tikz-cd}
\tikzset{
    lablrot/.style={anchor=north, rotate=-90, inner sep=1.5mm},
    lablrotsouth/.style={anchor=south, rotate=-90, inner sep=1.5mm}
}

\usepackage[hidelinks]{hyperref}
\usepackage{cleveref}

\usepackage{tikz}
\usepackage{standalone}
\usepackage{forest}

\numberwithin{equation}{section}

\theoremstyle{plain}
\newtheorem{theorem}[equation]{Theorem}
\newtheorem*{theorem*}{Theorem}

\newtheorem{proposition}[equation]{Proposition}
\newtheorem{lemma}[equation]{Lemma}
\newtheorem*{lemma*}{Lemma}
\newtheorem*{proposition*}{Proposition}
\newtheorem*{corollary*}{Corollary}

\newtheorem{corollary}[equation]{Corollary}

\newtheorem{maintheorem}{Theorem}

\theoremstyle{definition}

\theoremstyle{remark}
\newtheorem*{remark}{Remark}

\newcommand{\dd}{\mathrm{d}}
\newcommand{\BO}{\mathcal{O}}

\newcommand\HH{\mathrm{H}}

\newcommand{\hh}{\mathrm{h}}

\newcommand{\CC}{\mathbb{C}}
\newcommand{\ZZ}{\mathbb{Z}}

\newcommand{\QQ}{\mathbb{Q}}
\newcommand{\PP}{\mathbb{P}}

\newcommand\restr[1]{\raisebox{-.5ex}{$|$}_{#1}}
\newcommand{\sing}{\mathrm{sing}}
\newcommand{\Sing}{\mathrm{Sing}}

\newcommand{\sm}{\mathrm{sm}}

\newcommand{\GG}{\mathcal{G}}

\newcommand{\EE}{\mathscr{E}}

\newcommand{\AAA}{\mathbb{A}}

\DeclareMathOperator{\divv}{div}

\DeclareMathOperator{\Ext}{Ext}
\newcommand{\sheafhom}{\mathscr{H}\kern -.5pt om}

\DeclareMathOperator{\Id}{Id}

\newcommand{\IC}{\mathrm{IC}}

\newcommand{\Nm}{\mathrm{Nm}}
\newcommand{\Pic}{\mathrm{Pic}}

\DeclareMathOperator{\Prym}{Prym}

\newcommand{\dec}{\mathrm{dec}}

\newcommand{\ud}{{\underline{d}}}

\newcommand{\Nd}{{N_{\underline{d}}}}
\newcommand{\tNd}{{\tilde{N}_{\ud}}}
\newcommand{\PPP}{{\overline{\mathrm{P}}\mathrm{ic}^\ud(C)}}
\newcommand{\PPN}{\PPP_N}

\newcommand{\SE}{\mathscr{S}}
\newcommand{\BE}{\mathscr{BE}}

\newcommand{\Ker}{\mathrm{Ker}}

\newcommand{\mult}{\mathrm{mult}}

\DeclareMathOperator{\Bl}{Bl}

\usepackage{subfiles}

\usepackage{microtype} 
\usepackage[backend=biber, style=alphabetic, giveninits=true, maxalphanames=99, maxnames=99,  doi=false,isbn=false,url=false]{biblatex}
\newcommand{\ns}{\mathfrak{X}}

\newcommand{\Xiasing}{\Xi_{\sing,\mathrm{ad}}}
\newcommand{\Wasing}{W_{\sing,\mathrm{ad}}}
\newcommand{\Wtasing}{\tilde{W}_{\sing,\mathrm{ad}}}

\title{The boundary of the bielliptic Prym locus}
\author{Constantin Podelski}

\begin{document}
\maketitle
\begin{abstract}
We study the conormal geometry theta divisors of certain singular bielliptic curves. We apply these results to the boundary components $\SE_\ud$ of the bielliptic Prym locus. We obtain results on the Gauss map, compute the Chern-Mather class and the characteristic cycle of the intersection complex of the corresponding Prym theta divisor.
\end{abstract}

\tableofcontents

\section*{Introduction}
\label{sec:intro}
\addcontentsline{toc}{section}{\nameref{sec:intro}}

Let $\mathcal{A}_g$ denote the moduli space of $g$-dimensional principally polarized abelian varieties (ppav's for short) over the complex numbers. The \emph{biellitpic Prym} locus $\BE_g\subset \mathcal{A}_g$ is defined as the closure in $\mathcal{A}_g$ of the locus of Prym varieties of étale double covers of bielliptic curves, i.e. curves admitting a double cover to an elliptic curve $E$. In \cite{podelski2023GaussEgt}, they introduce the boundary components $\SE_\ud \subset \BE_g$ for $\ud=(d_1,\dots,d_n)$ with $\deg \ud=g$. These correspond to degenerations of the above situation where the ellitpic curve $E$ degenerates to an $n$-cycle of $\PP^1$'s, i.e. $E$ has $n$ components, which are rational, and its dual graph is the cyclic $n$-graph. By \cite[Th. 2]{podelski2023GaussEgt}, two cases are of particular interest from the point of view of the Schottky problem: When $\ud=(g)$ or $\ud=(1,g-1)$, for a general $(P,\Xi)\in \SE_\ud$, the degree of the Gauss map is the same as for non-hyperelliptic Jacobians (and these are the only values of $\ud$ for which this happens). In the present paper, we carry out a detailed study of the conormal geometry of the Prym theta divisor in these two cases. The results will enable us to show in a subsequent paper that the Tannakian representation associated to these ppav's (as in \cite{Kraemer2021MicrolocalGauss2}) differs from that of Jacobians. Note that Pryms in $\SE_\ud$ are \emph{never} Jacobians when $g\geq 4$ by \cite{Shokurov1984PrymVarieties}. \par 
Let $(A,\Theta)\in \mathcal{A}_g$ and $Z\subset A$ a subvariety. We define the \emph{conormal variety} to $Z$ by
\[ \Lambda_Z\coloneqq \overline{ \{(x,\xi)\in T^\vee A\,|\, x\in Z_\sm \,, \xi \bot T_x Z \} } \subset T^\vee A \,. \]
The projectivized conormal variety is the projectivization $\PP\Lambda_Z \subset \PP T^\vee A$. Translations induce a canonical trivialization of the cotangent bundle $T^\vee A=A\times V$, where $V\coloneqq T^\vee_0 A$. We define the Gauss map attached to $Z$ as the projection
\[ \gamma_Z: \PP\Lambda_Z \to \PP V \,. \]
Let $q:A\times \PP V \to A$ be the projection onto the first factor and $h\coloneqq c_1(\BO_{\PP V}(1))\in \HH^2(\PP V,\QQ)$. For $r\geq0$, the $r$-th Chern-Mather class of $Z$ is defined as
\[ c_{M,r}(Z)\coloneqq q_\ast \left( h^r\cap [\PP\Lambda_Z] \right) \in \HH_{2r}(A,\QQ) \,, \]
where $h$ is pulled back to $T^\vee A$ in the obvious way. The degree of the Gauss map is by definition the degree of the $0$-th Chern-Mather class. We extend the result of \cite[Th. 2]{podelski2023GaussEgt} to Chern-Mather classes of higher dimension:
\begin{maintheorem}[\ref{Thm: Prym: Chern-Mather of Prym theta}]\label{Maintheorem: Prym: Chern-Mather of Prym theta}
Let $(P,\Xi)\in \SE_{g}\cup \SE_{1,g-1}$, then the Chern-Mather classes of $\Xi$ are given by
\[ c_{M,r}(\Xi)= \frac{\xi^{g-r}}{(g-r)!} \binom{2g-2r-2}{g-r-1}\cap [P]\in \HH_{2r}(P,\QQ)  \]
for $r\geq 1$, where $\xi=c_1(\BO_P(\Xi))\in \HH^2(P,\QQ)$.
\end{maintheorem}
This coincides with the expected Chern-Mather classes for Jacobians. The Gauss map $\gamma_\Xi:\PP\Lambda_\Xi \to \PP V$ is not finite in general, but we have the following bound on the dimension of the locus above which finiteness fails:
\begin{maintheorem}[\ref{Thm: bielliptic Prym: fibers of Gauss map }]\label{Maintheorem: bielliptic Prym: fibers of Gauss map }
    Let $(P,\Xi)\in \SE_g\cup \SE_{1,g-1}$, then away from a subset $S\subset \PP V$ of codimension at least $3$, $\gamma_\Xi$ is finite.
\end{maintheorem}
 Recall that the characteristic cycle of the intersection complex $\mathrm{CC}(\IC_\Theta)$ is irreducible for a non-hyperelliptic Jacobian $(JC,\Theta)$ \cite{BresslerBrylinski97}. We have an analogous result for the loci $\SE_\ud$, apart from a correction term in the odd-dimensional case:
\begin{maintheorem}[\ref{Thm: Characteristic Cycle of Xi}]\label{Maintheorem: Characteristic Cycle of Xi}
Let $g\geq 4$ and $(P,\Xi)\in \SE_g\cup \SE_{1,g-1}$. If $g$ is even, then 
\[ \mathrm{CC}(\IC_\Xi)=\Lambda_\Xi\,. \]
If $g$ is odd, then
\[ \mathrm{CC}(\IC_\Xi)=\Lambda_\Xi+\sum_{x\in \Xiasing } 2 \Lambda_x\,, \]
where $\Xiasing$ is the set of isolated singularities of $\Xi$.
\end{maintheorem}
Note that for a general $(P,\Xi)\in \SE_\ud$, the set $\Xiasing$ is empty. We will give in Section \ref{Sec: Pryms associated to biellitpic curves} a very explicit description of $\Xiasing$ in terms of the ramification points of the double cover to $E$. By \ref{Cor: points in Xiasing a isolated singularities of maximal rank}, the points in $\Xiasing$ are isolated quadratic singularities of maximal rank of $\Xi$, which explains why they appear in the characteristic cycle depending on the parity of $g$. \par 
All of these results follow from the study of a particular type of singular bielliptic curves, namely those that admit a double cover $\pi:C\to E$ to a cycle of $\PP^1$'s, such that the corresponding involution on $C$ fixes the singular points but does not exchange the branches at these points. An essential tool in the study of theta divisors of smooth curves is the Abel-Jacobi map. We construct an analogue of the Abel-Jacobi map in our particular setting (which works more generally in the setting of cyclic curves). Note that if $\ud=(g)$, then $C$ is irreducible and our construction coincides over the locus of line bundles with already existing generalizations of the Abel-Jacobi map using the Hilbert scheme as in \cite{AltmanKleiman1980PicScheme}. In the reducible case, we replace the Hilbert scheme with a blowup of the symmetric product. This construction has the advantage of being very explicit and allows computations in cohomology.  \par 
In Section \ref{Sec: The bielliptic Prym locus} we recall the construction and basic properties of the loci $\SE_\ud$. In Section \ref{Sec: Singular bielliptic curves} we study singular bielliptic curves, construct the Abel-Jacobi map and derive some of its properties. We then apply the results to the Prym varieties in Section \ref{Sec: Pryms associated to biellitpic curves} and we prove Theorems \ref{Maintheorem: Prym: Chern-Mather of Prym theta}, \ref{Maintheorem: bielliptic Prym: fibers of Gauss map } and \ref{Maintheorem: Characteristic Cycle of Xi}. \par 
We work over the field of complex numbers.

\section{The bielliptic Prym locus}\label{Sec: The bielliptic Prym locus}
In this section we recall general facts on bielliptic Prym varieties. The references are \cite{Debarre1988}, \cite{donagi1992fibers} and \cite{Naranjo1992BiellipticPryms}. The curves considered will always be complete connected nodal curves over $\CC$. By the genus of a curve $C$ we mean the arithmetic genus
\[p_a(C)\coloneqq 1- \chi(C,\BO_C) \,. \] 
Let $\pi:\tilde{C}\to C$ be a double cover of nodal curves, corresponding to an involution $\sigma:\tilde{C}\to \tilde{C}$. At a nodal point $x\in \tilde{C}$, $\pi$ is of one of the following three types (see \cite{Beauville1977}):
\begin{enumerate}[1)]
    \item The involution doesn't fix $x$.
    \item The involution fixes $x$ and exchanges both branches.
    \item The involution fixes $x$ but preserves each branch. 
\end{enumerate}
For $I\subset \{1,2,3\}$, we say that $\pi$ is of type $I$ if it is of type $(i)$ for some $i\in I$ at every singular point of $\tilde{C}$. We sat $\pi$ is of type $(\ast)$ if $\pi$ is of type (3) and moreover étale away from the singular locus. This corresponds to the $(\ast)$ condition in \cite{Beauville1977}. We define the ramification divisor of $\pi$ as the Cartier divisor on $\tilde{C}$ defined by the exact sequence
\[ 0 \to \pi^\ast \omega_C \to \omega_{\tilde{C}} \to \BO_R \to 0\,, \] 
where $\omega_C$, $\omega_{\tilde{C}}$ are the dualizing bundles.\par 

The \emph{biellitpic Prym} locus $\BE_g\subset \mathcal{A}_g$ is defined as the closure in $\mathcal{A}_g$ of the locus of Prym varieties $(P,\Xi)=\mathrm{Prym}(\tilde{C}/C)$ where $\pi:\tilde{C}\to C$ is of type $(\ast)$, and $C$ is a curve of genus $g+1$ admitting a double cover $p:C\to E$ to a genus $1$ curve. Suppose that the Galois group of the composition $p\circ \pi$ is $(\ZZ/2\ZZ)^2$. Then the two other intermediate quotients induce a tower of curves
\begin{equation}\label{Fig: Cartesian Diagram of curves Egt} 
\begin{tikzcd}
   & \tilde{C} \arrow[dl,"\pi"'] \arrow[d,"\pi'"] \arrow[dr,"\pi''"] & \\
    C \arrow[dr,"p"'] & C' \arrow[d,"p'"] & C'' \arrow[dl,"p''"] \\
    & E & 
\end{tikzcd}\,.
\end{equation}
We can assume $g(C')=t+1 \leq g(C'')=g-t+1$ for some $0\leq t\leq g/2$. Denote by $\EE_{g,t}'$ the set of Pryms obtained in this way with the additional assumption that $E$ is smooth, and by 
\[\EE_{g,t}\coloneqq \overline{\EE_{g,t}'}\subset \mathcal{A}_g \] 
its closure in $\mathcal{A}_g$. It is well-known  \cite{Debarre1988} \cite{donagi1992fibers} that for $g\geq 5$, the loci $\EE_{g,t}$ for $0\leq t\leq g/2$ are the $\lfloor g/2 \rfloor $ irreducible components of $\BE_g$. The set of bielliptic Pryms where the Galois group of $p\circ \pi$ is $\ZZ/4\ZZ$ is contained in $\EE_{g,0}$ by \cite{donagi1992fibers}. In \cite{podelski2023GaussEgt}, they define the following subloci of $\EE_{g,t}$:
Suppose that we have the above situation, but that $E$ is an $n$-cycle of $\PP^1$'s, i.e. the $n$ irreducible components of $E$ are rational and the dual graph of $E$ is the cyclic $n$-graph. Let $\Delta\in E_{\sm,2g}$ be the branch locus of $p$ and $\ud\coloneqq \underline{\deg}(\Delta)/2$. Assume moreover that $\mathrm{Ram}(p')=0$. Then $\SE_\ud$ is defined as the set of Prym varieties $\Prym(\tilde{C}/C)$ obtained in this way. \par 
It turns out \cite[Lem. 4.21]{podelski2023GaussEgt}, that the two types of loci defined above cover all of $\BE_g$ apart from the intersection with the Jacobian locus $\mathcal{J}_g$ and the locus of decomposable ppav's $\mathcal{A}_g^\dec$
\[ \BE_g=\BE_g\cap(\mathcal{J}_g\cup \mathcal{A}_g^\dec) \cup \bigcup_{t=0}^{\lfloor g/2 \rfloor } \EE'_{g,t} \cup \bigcup_{\deg \ud = g} \SE_\ud \,. \]
The goal of the present paper is to carry out an in-depth study of the geometry of the Prym theta divisor for Prym varieties in $\SE_\ud$.\par 
Fix $\ud$ with $\deg \ud=g$ and $(P,\Xi)=\Prym(\tilde{C}/C)\in \SE_\ud$. We keep the notations of \ref{Fig: Cartesian Diagram of curves Egt}.  Because $\pi$ is of type $(3)$, $p$ is of type $(1,2)$ and $p''$ is of type $(3)$. Thus $C''$ has exactly one component sitting above each component of $E$. We can thus identify the multidegrees on both curves. As $p'$ is unramified we have $\mathrm{Branch}(p'')=\mathrm{Branch}(p)=\Delta$. Moreover $p$ is flat and we can thus associate to it the line bundle $\delta\coloneqq \det(p_\ast \BO_C)^{-1}\in \Pic^\ud(E)$ verifying $\Delta\in |\delta^{\otimes 2}|$ (see \cite[Sec. 4]{podelski2023GaussEgt}). Let
\begin{equation}\label{Equ: Definition of P'' Xi''}
\begin{aligned} 
P'' &\coloneqq \{L\in \Pic^\ud(C'') \,|\, \Nm_{p''}(L)=\delta\} \,,  \quad \text{and} \\
\Xi''&\coloneqq \Theta''\cap P'' \,, 
\end{aligned}
\end{equation}
where $\Theta''\coloneqq \{L\in \Pic^{\ud}(C'')\,|\, \hh^0(L)>0 \} \subset\Pic^\ud(C'')$ is the theta divisor as defined in \cite{Beauville1977}. Note that $(P'',\Xi'')$ is not principally polarized but of type $(1,\dots,1,2)$. By \cite[Prop. 4.12, Prop. 4.14]{podelski2023GaussEgt}, the pullback induces a degree $2$ isogeny
\[ \pi''^\ast: P''\to P \,, \quad \text{with } \quad (\pi''^\ast)^\ast \Xi=\Xi'' \,. \]
This isogeny is the quotient by the two-torsion point $p''^\ast \delta (-R'')\in JC''$, where $R''$ is the ramification divisor of $p''$. Thus the geometry of $(P,\Xi)$ can be completely understood through the study of $(P'',\Xi'')$. Since $p'':C''\to E$ is a morphism to a (singular) elliptic curve, it is substantially easier to study than $\pi$. Another big advantage is that $\Xi''$ is defined as a scheme-theoretic intersection, in contrast to $\Xi$ who is defined as an intersection up to multiplicity only. Thus from now on, we will forget about the double covering $\pi:\tilde{C}\to C$, and study the following situation:
\begin{itemize}
    \item A nodal curve $C''$ with a double covering $p'':C''\to E$ of type (3), where $E$ is a cycle of $\PP^1$'s.
    \item A fixed $\delta\in \Pic^\ud(E)$ with $\Delta\in |\delta^{\otimes 2}|$ where $\Delta$ is the branch divisor of $p''$.
\end{itemize}
We define $(P'',\Xi'')$ by \ref{Equ: Definition of P'' Xi''}, and obtain a principally polarized abelian variety after quotienting by $\langle p''^\ast\delta(-R'')\rangle $. Thus Theorem \ref{Maintheorem: Prym: Chern-Mather of Prym theta}, \ref{Maintheorem: bielliptic Prym: fibers of Gauss map } and \ref{Maintheorem: Characteristic Cycle of Xi} will follow immediately from Theorem \ref{Thm: Prym: Chern-Mather of Prym theta}, \ref{Thm: bielliptic Prym: fibers of Gauss map } and \ref{Thm: Characteristic Cycle of Xi} respectively.

\section{Singular bielliptic curves}\label{Sec: Singular bielliptic curves}
\subsection{Preliminaries}
We start by setting some notations. From now on, $E$ will be a cycle of $\PP^1$'s, i.e. the normalization of $E$ has $n$ components $E_1,\dots,E_n$ each isomorphic to $\PP^1$ and the dual graph of $E$ is a cyclic graph. We assume that for $i\in \ZZ/n\ZZ$, $E_i$ intersects $E_{i-1}$ at $Q_i^0$ and $E_{i+1}$ at $Q_i^\infty$. Let $Q_i\in E$ be the image of $Q_i^0$ (i.e. the intersection of $E_i$ and $E_{i-1}$. We fix an identification of $E_i$ with $\PP^1$ where $Q_i^0$ is identified with $0$, $Q_i^\infty$ is identified with $\infty$. This identification also gives coordinates near $Q_i^0$ and $Q_i^\infty$ coming from $\PP^1$. With these coordinates we can identify the group of Cartier divisors supported at $Q_i$ with $\CC^\ast \times \ZZ \times \ZZ$ (see \cite{Beauville1977}). \par 
Let $C$ be a stable nodal curve of genus $g+1$ with a double covering $\pi:C\to E$ of type (3). It follows that the associated involution $\tau:C\to C$ preserves each irreducible component, fixes the singular points and is not the identity on any component. This implies that $C$ has $n$ component, one above each component of $E$ and the dual graph is cyclic. Let $\beta: N_1\cup\cdots \cup N_n \to C$ be the normalization, and $\pi_i:N_i\to E_i\simeq \PP^1$ be the induced morphism. 
\[ 
\begin{tikzcd}
    C \arrow[d,"\pi"'] & N \arrow[l,"\beta"'] \arrow[dl,"\pi_N"'] & N_i \arrow[l,hook'] \arrow[d,"\pi_i"] \\
    E & E_i \arrow[l,hook'] \arrow[r, phantom, "\simeq "] & \PP_1
\end{tikzcd}
\]
Let $P_i^0$ (resp. $P_i^\infty$) be the point in $N_i$ sitting above $Q_i^0$ (resp. $Q_i^\infty$). By assumption the morphism $\pi_i$ is ramified at $R_i+P_i^0+P_i^\infty$ for some $R_i\subset N_{i}$. Let
\begin{align*}  
R&\coloneqq R_1+\cdots+R_n \subset C\,, \\
\Delta &\coloneqq \pi_\ast(R) \subset E\,,\\ 
\ud&\coloneqq (d_1,\dots,d_n)\,, \quad \text{with $d_i\coloneqq \deg R_i /2$\,. } 
\end{align*}
Both $R$ and $\Delta$ are reduced an non-singular divisors, and are the ramification and branch locus of $\pi$, respectively. Let $\delta_i$ be the hyperelliptic bundle on $N_i$.
There is an exact sequence
\begin{equation}\label{Equ: exact sequence of biellitpic jacobian version 1} 0\to \CC^\ast \to \Pic^{\underline{d}}(C) \overset{\beta^\ast}{\longrightarrow} \Pic^{\underline{d}}(N) \to 0 \,. 
\end{equation}
\begin{proposition}
C is hyperelliptic if and only if $d=(1)$ or $d=(1,1)$.
\end{proposition}
\begin{proof}
We use the usual notion of hyperellipticity for singular nodal curves (see \cite[p. 101]{arbarello2}). The stability assumption on $C$ implies $d_i>0$ for all $i$. Suppose $C$ is hyperellitpic and let $\sigma:C\to C$ be the hyperelliptic involution. If $\sigma$ exchanges components, these must we isomorphic to $\PP^1$, contradicting the stability of $\CC$. Thus $\sigma$ preserves the components. $\sigma$ can't be of type (3) at any node because then the node would have to be separating. Thus the only possibilities left are
\begin{itemize}
    \item $C$ has one component, $\sigma$ exchanges both branches at the node, thus $N/\sigma=\PP^1$ but $\sigma\neq \tau $ (since $\tau\restr{N}$ preserves $P_1^0$ and $P_1^\infty$ and $\sigma$ exchanges them). Since the hyperellitpic involution is unique when the $p_a(N)>1$, this implies $d=(1)$.
    \item $C$ has two components, and the two nodes are exchanged by $\sigma$. Again $\sigma\restr{N}$ exchanges $P_1^0$ and $P_1^\infty$ (resp. $P_2^0$ and $P_2^\infty$). Thus $\sigma\neq \tau$, thus $p_a(N_1)=p_a(N_2)=1$ and $d=(1,1)$.
\end{itemize} 
\end{proof}
Recall the Brill-Noether varieties defined by
\[ W^r_\ud (C) \coloneqq \{L\in |\Pic^\ud (C) \,|\, \hh^0(C,L)\geq r+1 \}\subset \Pic^\ud(C) \,. \]
We have the following ``Martens Theorem" type result
\begin{proposition}\label{Prop: Martens Thoerem for singular bielliptic curves}
Suppose $0\leq \ud \leq \underline{\deg}(\omega_C)/2$, and $0<2r\leq d$. Then
\[ \dim W^r_\ud(C) = d-2r-1 \,. \]
\end{proposition}
\begin{proof}
Recall there is an exact sequence
\[ 0 \to \BO_C \to \beta_\ast \BO_N \to \bigoplus_{i=1}^n \CC_{Q_i} \to 0 \,, \]
where $Q_i$ are the singular points of $C$. From this we derive the exact sequence
\[ 0 \to \HH^0(C,L)\to \HH^0(N,\beta^\ast L) \overset{\psi}{\longrightarrow} \CC^n \,. \]
$\psi$ depends on the gluing of the line bundle above the nodes, and on whether or not the $P_i^0$ and $P_i^\infty$ are base points for $|L\restr{N_i}|$. We will now give an explicit basis of this $\Ker (\psi)$. Let $\Gamma$ be the dual graph of $C$. For each $i$, if $\hh^0(N_i,L\restr{N_i})=0$, delete the vertex $(N_i)$ and the edges to it from $\Gamma$. If $\hh^0(N_i,L\restr{N_i})>0$ we can write in a unique way
\[ L\restr{N_i}=k_i \delta_i + \BO_{N_i}(a_i^0 P_i^0+a_i^\infty P_i^\infty + D) \,, \]
where $a_i^0,a_i^\infty \in \{0,1\}$, and $D$ is $\tau$-simple (recall that $2P_i^0\sim 2P_i^\infty \sim \delta_i$). Note also that if $g(N_i)=1$, then $d_i\leq 1$ by assumption. We have $\hh^0(N_i,L\restr{N_i})=k_i+1$.
\par 
If $a_i^0=a_i^\infty=1$, then $\HH^0(N_i,L\restr{N_i})\subset \Ker (\psi)$. Delete the vertex corresponding to $N_i$ and the edges to this vertex from the graph $\Gamma$. Else, if $k_i=0$, and $a_i^0=1$ (resp. $a_i^\infty=1$), mark the vertex $(i)$ with $\{0\}$ (resp. $\{\infty\}$). In this case, the space of sections is generated by a section vanishing at $P_i^0$ (resp. $P_i^\infty$) and not at $P_i^\infty$ (resp. $P_i^0$) (if $a_i^=a_i^\infty=0$ mark the vertex with $ \{0,\infty\}$). If $k_i=0$ and $a_i^0=a_i^\infty=0$ the space of sections is generated by a section vanishing neither at $P_i^0$ nor at $P_i^\infty$. \par 
If $k_i>0$, and $a_i=0=a_i^\infty =0$, then mark the vertex $(i)$ with $\{0,\infty \}$. In this case, $\HH^0(N_i,L\restr{N_i})$ is generated by $k_i-1$ sections vanishing at $P_i^0$ and $P_i^\infty$, a section vanishing at $P_i^0$ but not $P_i^\infty$ and a section vanishing at $P_i^\infty$ but not $P_i^0$. \par 
If $k_i>0$ and $a_i^0=1$ but $a_i^\infty = 0$ (resp. the opposite), mark the vertex $(i)$ with $\{0\}$ (resp. $\{\infty\}$). \par 
The space of sections of $L$ supported on single component is of dimension
\[ \bigoplus_{i=1}^n \hh^0(N_i,L\restr{N_i}(-P_i^0-P_i^\infty))=\sum_{i=1}^n \max(0, k_i+a_i^0+a_i^\infty -1) \,. \]
The other sections are generated in the folowing way: start with a vertex of $\Gamma$ marked with $0$. that corresponds to a section $s_i^0$ vanishing at $P_i^0$ but non-zero at $P_i^\infty$. This imposes a non-zero value at $P_{i+1}^0$. If $(i+1)$ is marked with $\infty$ we have a new section. If $(i+1)$ isn't marked at all that imposes a coefficient on the section of $\HH^0(N_{i+1},L\restr{N_{i+1}})$ not vanishing at $P_{i+1}^0$ and $P_{i+1}^\infty$ and we move on to $(i+2)$. If $(i+1)$ is marked with $0$ but not $\infty$ it is impossible to complete $s_i^0$ to a section. We repeat this process on the whole graph.\par   
Thus, for each segment $(i,i+1,\dots,j)$ of the modified graph $\Gamma$ such that $i$ is marked with $0$, $j$ is marked with $\infty$, and $i+1,\dots,j-1$ are not marked, there is an additional section. In particular, since increasing $k_i$ by one or having a pair $(P_i^0,P_j^\infty)$ imposes a condition of codimension two on $W^0_d(N)$, we see that 
\[ \dim W^r_d(C) \leq d-2r-1 \]
Moreover it is clear that choosing $a_i^0$ and $a_i^\infty$ properly we achieve this bound.
\end{proof}
\begin{remark}
Our proof also shows that for $r\geq 1$, the varieties $W^r_\ud(C)$ are the preimages by $\beta^\ast$ of certain varieties in $\Pic^\ud(N)$.
\end{remark}

As in \cite[Sec. 2]{Beauville1977}, we define the $\Theta$-divisor in $\Pic^{\underline{d}}(C)$ by
\[ \Theta \coloneqq \{ L \in \Pic^{\underline{d}}(C)\,|\, \hh^0(C,L)>0 \} \subset \Pic^{\underline{d}}(C)\,. \]
Recall the Riemann Singularity Theorem, who is due in this form to Kempf in the irreducible context and Beauville \cite[Prop. 3.11]{Beauville1977} in the reducible context:
\begin{proposition}\label{Prop: Riemann Singularity Theorem fro reducible curves}
    Let $C$ be a connected nodal curve, and assume $\underline{\deg}(\omega_C)=2\ud$ is even. Let $L\in \Theta=\{M\in \Pic^{\underline{d}}(C)\,|\, \hh^0(M)>0 \}$ and consider the pairing
    \begin{align*}
       \phi: \HH^0(C,L) \otimes \HH^0(C,\omega_C-L) &\to \HH^0(C,\omega_C)\,.
    \end{align*}
    Let $(s_i)$ and $(t_j)$ be a basis of $\HH^0(C,L)$ and $\HH^0(C,\omega_C-L)$ respectively. Then
    \[ \mathrm{mult}_L \Theta \geq \HH^0(C,L) \,, \]
    with equality if and only if $\det(\phi(s_i\otimes t_j))$ is non-zero, in which case it gives the tangent cone of $\Theta$ at $L$.
\end{proposition}
We will call a singularity of $\Theta$ \emph{exceptional} if the equality doesn't hold above. We define a relation on triples $(i,j,k)$ by
\begin{align*}  i \prec j \prec k &\iff \begin{cases}
    i< j <  k\,, \quad  \text{or} \\
    i\geq k\,, \quad \text{and} \quad j\notin [[k,i]]\,.
\end{cases} \,.\\
i \preceq j \prec k &\iff  (i<j<k) \quad \text{or} \quad (i=j \text{ and } j\neq k) \,.
\end{align*}
We also define $i \prec j \preceq k$ and $i\preceq j \preceq k$ in the obvious way. We then have
\begin{proposition}\label{Prop: Singular Locus of theta for bielliptic Curves}
The singular locus of $\Theta$ is
\[ \Sing(\Theta)=(\beta^\ast)^{-1}(\mathcal{A}\cup \mathcal{B} \cup \mathcal{B}'\cup \mathcal{C} \cup \mathcal{C}' )\,, \]
where
\begin{align*}
    \mathcal{A}&= \bigcup_{i,j}\{ \delta_i +\delta_j +\alpha_N(N_{\underline{d}-2e_i-2e_j} )\}\,,\\ 
    \mathcal{B}&= \bigcup_{i\preceq j\preceq k} \{ \BO_N(P_i^0+P_k^\infty)+\delta_j + \alpha_N(N_{\underline{d}-e_i-2e_j-e_k}) \} \,, \\
        \mathcal{B}'&= \bigcup_{i\prec j\prec k} \{ \BO_N(P_i^\infty+P_k^0)+\delta_j + \alpha_N(N_{\underline{d}-e_i-2e_j-e_k}) \} \,, \\
    \mathcal{C}&= \bigcup_{i\preceq j \prec k \preceq l } \{\BO_N(P_i^0+P_j^\infty+P_k^0+P_l^\infty)+\alpha_N(N_{\underline{d}-e_i-e_j-e_k-e_l}) \} \,,\\
     \mathcal{C}'&= \bigcup_{i\prec j \preceq k \prec l } \{\BO_N(P_i^0+P_j^0+P_k^\infty+P_l^\infty)+\alpha_N(N_{\underline{d}-e_i-e_j-e_k-e_l}) \} \,.
\end{align*}
A general point of $(\beta^\ast)^{-1}(\mathcal{A}\cup \mathcal{B}\cup \mathcal{C})$ is not an exceptional singularity, and a general point in $(\beta^\ast)^{-1}(\mathcal{B}'\cup\mathcal{C}')$ is an exceptional singularity. More precisely, if $\underline{d}=(g)$, then $\mathcal{B}'=\mathcal{C}=\mathcal{C}'=\emptyset$, all singularities are non exceptional, and we have
\[ \Sing_k(\Theta)\coloneqq \{x\in \Theta\,|\, \mult_x \Theta \geq k\}=(\beta^\ast)^{-1}(\mathcal{A}_k\cup \mathcal{B}_k)\]
with 
\begin{align*} 
\mathcal{A}_k&=\{k\delta_N+\alpha_N(N_{g-2k}) \}\,, \\
\mathcal{B}_k&=\{ \BO_N(P^0+P^\infty)+(k-1)\delta_N+\alpha_N(N_{g-2k})\} \,. 
\end{align*}
If $\underline{d}=(1,g-1)$, then $\mathcal{B}'=\mathcal{C}=\mathcal{C}'=\emptyset$, all singularities are non exceptional, and we have
\[ \Sing_k(\Theta)=(\beta^\ast)^{-1}(\mathcal{A}_k\cup \mathcal{B}_k)\]
with 
\begin{align*} 
\mathcal{A}_k&=\{ \Pic^1(N_1)+k\delta_{2}+\alpha_{N_2}(N_{2,g-2k-1}) \}\,, \\
\mathcal{B}_k&=\{\Pic^1(N_1)+ \BO_{N_2}(P_2^0+P_2^\infty)+(k-1)\delta_{2}+\alpha_{N_2}(N_{2,g-2k-1})\} \\
&\quad \cup \{ \BO_N(P_1^0+P_2^\infty)+(k-1)\delta_2+ \alpha_{N_2}(N_{2,g-2k})\} \\
&\quad \cup \{ \BO_N(P_1^\infty+P_2^0)+(k-1)\delta_2+ \alpha_{N_2}(N_{2,g-2k})\}\,. 
\end{align*}

\end{proposition}
\begin{proof}
By \ref{Prop: Riemann Singularity Theorem fro reducible curves}, a point $L\in \Sing(\Theta)$ either verifies $\hh^0(C,L)\geq 2$ or $\hh^0(C,L)=1$ and $st=0$ where $\HH^0(C,L)=\langle s\rangle $ and $\HH^0(C,\omega_C-L)=\langle t \rangle$. From the proof of \ref{Prop: Martens Thoerem for singular bielliptic curves} it is clear that any $L$ with $\hh^0(C,L)\geq 2$ has to be in $\mathcal{A},\mathcal{B}$, or $\mathcal{C}$. It is also straightforward to check that an exceptional singularity has to be in $\mathcal{B}'$ or $\mathcal{C}'$: indeed if $L$ is an exceptional singularity, there is a section $s\in \HH^0(C,L)$ such that $s\cdot \HH^0(C,\omega_C-L)=0$. This implies that we can find $i,j $ such that $s$ is zero say on components $l$ with $i\prec l \prec j$ and non-zero on $i$ and $j$. There are also $i',j'$ such that $\HH^0(C,\omega_C-L)$ is supported on components $l$ with $i'\prec l \prec j'$, and 
\[ i \preceq i' \preceq j' \preceq j \,. \]
If $i'=j'$ we are in $\mathcal{B}'$, else we are in $\mathcal{C}'$.
\par
When $\underline{d}=(g)$ or $\underline{d}=(1,g-1)$, $\mathcal{B}'=\mathcal{C}=\mathcal{C}'=\emptyset$ for degree reasons. It is immediate that any line bundle in $\mathcal{A}_k$ or $\mathcal{B}_k$ is not exceptional and thus the assertion about the multiplicities of the singularities follow from \ref{Prop: Riemann Singularity Theorem fro reducible curves}.

\end{proof}

\subsection{The Abel-Jacobi map}\label{Sec: The Abel-Jacobi Map}
Recall that sections of $\omega_{C}$ are $1$-forms $\omega$ on $N$ which can have poles at $P^0_i$ and $P^\infty_i$, subjected to the conditions
\begin{equation}\label{Equ: sections of omega C are section of omega N whose residue is zero} \mathrm{Res}_{P^\infty_i}\omega + \mathrm{Res}_{P^0_{i+1}} \omega = 0 \,, \quad \text{for} \quad i\in \ZZ/n\ZZ \,. 
\end{equation}
We thus have an inclusion of $\BO_{C}$-modules
\[ \beta_\ast \omega_N \subset \omega_{C''} \subset \beta_\ast \omega_N(\sum_i P^0_i+P^\infty_i) \,. \]
From what precedes we have
\begin{equation}\label{Equ: embedding of HH0 N omegaN into HH0 C omegaC} \HH^0(N,\omega_N) \subset \HH^0(C,\omega_{C}) \subset \HH^0(N,\omega_N(\sum_{i=1}^n P_i^0+P_i^\infty ))\,. 
\end{equation}
 Let $s_E$ be a generator of $\HH^0(E,\omega_E)$. As a $1$-form, $s_E$ is given on $E_i$ by $dz/z$ for a coordinate $z$ centered at $0$. Let $s_{R}=p^\ast s_E$ be the pullback as a $1$-form. $\pi_i:N_i\to E_i$ is ramified at $R_i+P_i^0+P_i^\infty$ thus $\divv (s_{R})=R$ as a section of $\omega_{C}$. For dimension reasons we have
\[ \HH^0(C,\omega_{C})=\HH^0(N,\omega_N)\oplus \langle s_{R} \rangle \,. \]
We see from the above discussion that $\HH^0(N,\omega_N)$ (resp. $\langle s_{R}\rangle$) is the -1 (resp. +1) eigenspace for the action of $\tau$ on $\HH^0(C,\omega_{C})$. We define
\begin{align*}
    |\omega_C|\coloneqq  \PP\HH^0(C,\omega_C) \,, \quad 
    |\omega_C|^-\coloneqq  \PP \HH^0(C,\omega_C)^- \,, \quad 
    |\omega_N| \coloneqq  \PP \HH^0(N,\omega_N) \,.
\end{align*}
We define a divisor to be singular if it intersects with the singular locus. The following lemma is very simple, but crucial:
\begin{lemma}\label{Lem: H in omegaC is singular iff H in omegaN}
    With the above notations, a divisor $H\in |\omega_C|$ is singular if and only if $H\in |\omega_C|^-$, and in that case
    \[ \sum_{i=0}^\infty P_i^0+P_i^\infty \leq \beta^\ast H\,.\]
\end{lemma}
\begin{proof}
    Let $H=\divv ( \lambda s_R+s)\in |\omega_C|$, where $s\in \HH^0(C,\omega_C)^-$ and $\lambda\in \CC$. By what precedes, $s$ comes from a section of $\HH^0(N,\omega_N)$. Sections of $\omega_N$ are holomorphic $1$-forms, thus immediately verify \ref{Equ: sections of omega C are section of omega N whose residue is zero}. As sections of $\omega_C$, they vanish at the singular points. $s_R$ is non-zero at the singular points thus $H$ is singular if and only if $\lambda=0$. In that case, $H$ vanishes at all the singular points.
\end{proof}
We thus have a canonical identification $\rho: |\omega_C|^- \overset{\sim}{\longrightarrow} |\omega_N|$ corresponding on the locus of non-singular divisors to
\begin{equation}\label{Definition of rho: omega C to omega N}
    \rho(H)= \beta^\ast H - \sum_{i=1}^n (P_i^0+P_i^\infty) \,.
\end{equation}

The Abel map is well known in the case of smooth, or singular irreducible curves. But for singular reducible curves the situation is much more technical. We will now show how to construct a candidate for the Abel map in the case of cyclic curves. In that case $JC$ sits in an exact sequence
\[ 0 \to \CC^\ast \to JC \to JN \to 0 \]
It is well known (see \cite{Serre1988AlgGroupsClassFields}) that
\[ \Ext(JN,\CC^\ast)\simeq \widehat{JN} \simeq JN\,, \]
and that under this identification, by \cite[Cor 12.5]{OdaSeshadri1979GenJac}, the extension defining $JC$ corresponds to the line bundle
\[ \eta \coloneqq \BO_N (\sum_{i=1}^n P_i^0-P_i^\infty) \in JN \,.\]
The corresponding line bundle on $JN$ is
\[ L^\eta\coloneqq \mathscr{L}\otimes  \tau_\eta \mathscr{L}^{-1}=\tau_{\eta_0}\mathscr{L}\otimes \tau_{\eta_\infty}\mathscr{L}^{-1}\in \widehat{JN}\,, \]
where $\mathscr{L}$ is the principal polarization on $JN$, $\tau_x$ is the translation by $x$ and
\[ \eta_0\coloneqq \BO_N(P_1^0+\cdots+P_n^0)\,, \quad \eta_\infty \coloneqq  \BO_N(P_1^\infty+\cdots+P_n^\infty) \,. \]
The corresponding extension is 
\[ JC \simeq L^\eta\setminus JN \]
where $JN\hookrightarrow L^\eta$ embeds as the $0$ section. We define 
\[ \overline{JC} \coloneqq \PP(L^\eta\oplus \BO_{JN} )= \PP(\tau_{\eta_0} \mathscr{L} \oplus \tau_{\eta_\infty} \mathscr{L}) \]
be the associated $\PP^1$-bundle. $\tau_{\eta_0}\mathscr{L}$ and $\tau_{\eta_\infty}\mathscr{L}$ canonically define bundles on $\Pic^\ud(N)$, thus we will see $\PPP$ as a $\PP^1$-bundle on $\Pic^\ud(N)$ from now on. This is of course not the usual compactification of the Picard scheme, but this will be the convient compactification for our computations. Let 
\[ \alpha_N:N_\ud \to \Pic^\ud(N)\]
be the Abel-Jacobi map, where 
\[ N_\ud \coloneqq N_{1,d_1}\times \cdots N_{n,d_n} \]
is the product of the symmetric product of the curves $N_1,\dots,N_n$. We have for $k\in \{0,\infty\}$
\[ \alpha_N^\ast \tau_{\eta_k} \mathscr{L} = \BO_{N_\ud}(B^k) \,, \quad \text{with} \quad B^k \coloneqq \sum_{i=1}^n (P_i^k+N_{i,d_i-1})\prod_{j\neq i} N_{j,d_j}\,. \]
Let $s^0,s^\infty$ be the sections on $\Nd$ corresponding to $B^0$ and $B^\infty$ respectively. Let
\[ \PPN\coloneqq \Nd\times_{\Pic^\ud(N)} \PPP \,. \]
We have the following commutative diagram
\begin{equation}\label{Tikzcd: definition of tNd}
\begin{tikzcd}
    \PPN \arrow[d] \arrow[r] & \PPP\arrow[d]\\
     \Nd \arrow[u,dashed, bend left=30, "{(s^0 , s^\infty)}"] \arrow[r,"\alpha_N"'] & \Pic^\ud(N)
\end{tikzcd}
\end{equation}
Let $b:\tilde{N}_{\underline{d}}\coloneqq \Bl_B N_{\underline{d}}\to \Nd $ be the blowup at $B\coloneqq B^0\cap B^\infty$. This resolves the indeterminancy of $(s^0,s^\infty)$
\begin{equation}\label{Diag: Def of tNd with more details}
    \begin{tikzcd}
    \tilde{N}_{\underline{d}} \arrow[dr ,"b"'] \arrow[rr, bend left=20, "\alpha"] \arrow[r,hook,"i_\tNd"'] 
         & \PPN  \arrow[r] \arrow[d,"q_N"] & \PPP \arrow[d,"q"]\\  
         & N_{\underline{d}}  \arrow[r,"\alpha_N"'] & \Pic^\ud(N)
    \end{tikzcd}\,,
\end{equation} 
and $\alpha$ is the Abel-Jacobi map we were looking for. By standard intersection theory we have
\begin{equation}\label{Equ: Class of tilde Nd (blowup of Nd)}
\begin{split}
    [\tNd]&=x_1+\cdots+x_n+h'\in \HH^2\left(\PPN,\QQ\right) \,,
\end{split}
\end{equation}
where $h'=c_1(\BO_\PPP(1))\in \HH^2(\PPP,\QQ)$ is the hyperplane section coming from the $\PP^1$-bundle structure and $x_i=[N_{i,d_i-1}]\in \HH^2(N_{i,d_i},\QQ)$ (we make the abuse of notation of omitting the pullback notation when it is clear). For $k\in\{0,\infty\}$ let
    \[ B_i^k \coloneqq P_i^k+ N_{\ud-e_i} \subset N_\ud \,,\]
    and $s_i^k\in \HH^0(\Nd,\BO_\Nd(B_i^k))$ the corresponding section. By definition we have
    \[ B=\bigcup_{i,j} B_i^0\cap B_j^\infty \,. \]
    In particular, locally $\tNd$ is defined inside $\PPN$ by the vanishing of
    \begin{equation}\label{Equ: equation defining tNd}
    \lambda s_1^0 \cdots s_n^0-\mu s_1^\infty \cdots s_n^\infty \,, \end{equation}
    where $(q_N,\lambda:\mu):\PPN\restr{U} \to U\times \PP^1$ is a local trivialisation of the $\PP^1$-bundle on an open set $U\subset \Nd$.
\begin{lemma}\label{Lem: singularities of tNd}
    Above non-singular divisors, $\tNd$ is smooth. Let $\tilde{D}=(D,\lambda:\mu)\in \tNd$ be a point above a singular divisor, where we use the notations of \ref{Equ: equation defining tNd}. Let
    \begin{align*}
        k\coloneqq &\# \{i\,|\, P_i^0\leq D\} + \delta_{\lambda,0}\,, \\
        l\coloneqq &\# \{i\,|\, P_i^\infty\leq D\} + \delta_{\mu,0}\,,
    \end{align*}
    where $\delta_{\lambda,0}=1$ if $\lambda=0$ and $0$ otherwise. We then have a local analytic isomorphism
    \[ (\tNd,\tilde{D}) \simeq (V(x_1 x_2\dots x_k-x_{k+1}x_{k+2}\cdots x_{k+l}),0)\subset (\AAA^{g+1},0)\,. \]
\end{lemma}
\begin{proof}
    Above non-singular divisors, the blowup $b:\tNd \to \Nd$ is a local isomorphism, thus $\tNd$ is smooth. Above singular divisors, the assertion follows from \ref{Equ: equation defining tNd} and the fact that for $k\in \{0,\infty \}$ and $1\leq i\leq n$, the divisors $B^k_i=\divv s^k_i$ are smooth normal crossing divisors on $\Nd$.
\end{proof}
Let $\overline{\Theta}$ be the closure of $\Theta$ in $\PPP$. Clearly we have a surjection $\alpha:\tNd \to \overline{\Theta}$. Although $\alpha$ is not a resolution of singularities, the singularities of $\tNd$ are much simpler than those of $\overline{\Theta}$.

\subsection{The conormal variety to theta}
Recall that the (projectivised) conormal variety is defined by
\[ \PP\Lambda_\Theta \coloneqq \overline{ \{ (x,H)\in \PP T^\vee JC \,|\, x\in \Theta_{\mathrm{sm}}\,, \, T_x \subset \Ker H\} } \subset \PP T^\vee JC  \,.\]
Since the cotangent space to $JC$ is trivial and canonically identified with $JC\times \HH^0(C,\omega_C)$, we will from now on view $\PP\Lambda_\Theta$ inside $JC\times |\omega_C|$. We define the projections
\[ 
\begin{tikzcd} 
N_{i,d_i} & \Nd \arrow[l] & \Nd\times |\omega_C|  \arrow[l,"p"] \arrow[ll,bend right=15, "p_{i}"'] \arrow[r,"\gamma"] & {|\omega_C|} 
\end{tikzcd} \quad \text{for all $1\leq i \leq n$}\,. 
\] 
We definine $\PP\Lambda_{\Nd}\subset \Nd\times |\omega_C|$ as the vanishing locus (i.e. the $0$-th determinantal variety) of the following composition of maps of vector bundles
\[ \gamma^\ast \BO_{|\omega_C|}(-1) \hookrightarrow \HH^0(C,\omega_C)\hookrightarrow \bigoplus_{i=1}^n \HH^0(N_i,\omega_{N_i}(P_i^0+P_i^\infty)) \overset{\oplus  \mathrm{ev}_i}{\longrightarrow} \bigoplus_{i=1}^n p_{i}^\ast E_{K,i} \,, \]
where the vector spaces are identified with the corresponding trivial vector bundles, and $E_{K,i}$ are the evaluation bundles on $N_{i,d_i}$ associated to the line bundle $\omega_{N_i}(P_i^0+P_i^\infty)$, and $\mathrm{ev}_i$ are the evaluation maps (see \cite[p. 339]{arbarello} for the definition of $E_{K,i}$). Thus set-theoretically we have
\begin{equation}\label{Equ: set-theoretic description of Lambda Nd}
    \PP\Lambda_{N_\ud}=\left\{ (D,H)\in \Nd\times |\omega_C|\,\big|\, D\leq \beta^\ast H \right\}\,.
\end{equation}
By \cite[p. 340]{arbarello}, for all $r\geq 0$, we have
\begin{equation}\label{Equ: chern class of E K,i (singular cas)}
    c_r(E_{K,i})=\sum_{k=0}^r \binom{r}{k} x_i^k \frac{\theta_i^{r-k}}{(r-k)!} \in \HH^{2r}(N_{i,d_i})\,.
\end{equation}
We also make the following computations: using Poincaré's Forumla \cite[p. 25]{arbarello} we have
\begin{equation}\label{Equ: alpha pushforward classes of E K,i (singular case)}
    \begin{split}
        \alpha_{N_i,\ast}(c_r(E_{K,i}))&=\frac{\theta_i^r}{r!} \sum_k \binom{r}{k} \binom{r}{k}=\frac{\theta_i^r}{r!} \binom{2r}{r}\in \HH^{2r}(JN_i,\QQ)\,, \\
        \alpha_{N_i,\ast}(x_i c_r(E_{K,i}))&=\frac{\theta_i^{r+1}}{(r+1)!} \sum_k \binom{r}{k} \binom{r+1}{k+1} = \frac{\theta_i^{r+1}}{(r+1)!} \binom{2r+1}{r+1}\,, \\
        \alpha_{N_i,\ast}(x_i^2 c_r(E_{K,i}))&= \frac{\theta_i^{r+2}}{(r+2)!} \binom{2r+2}{r+2} \,.
    \end{split}
\end{equation}

We have the following
\begin{proposition}\label{Prop: Lambda Nd is irreducible of dim g}
    Suppose $g\geq 3$, then the projection $\PP\Lambda_\Nd\to \Nd$ is birational. In particular $\PP\Lambda_\Nd$ is irreducible of dimension $g$.
\end{proposition}
\begin{corollary}\label{Cor: Class of Lambda Nd}
We have\begin{align*}
    [\PP\Lambda_\Nd]&= c_g( \gamma^\ast \BO_{|\omega_C|}(1) \otimes \bigoplus_{i=1}^n p_i^\ast E_{K,i} ) \\
    &= \sum_{r=0}^g h^r c_{g-r}(\bigoplus_{i=1}^n p_i^\ast E_{K,i} ) \in \HH_{2g}(\Nd\times |\omega_C|,\QQ)\,.
\end{align*}
\end{corollary}
\begin{proof}[Proof of the Corollary]
    The corollary follows from intersection theory.
\end{proof}
\begin{proof}[Proof of the Proposition]
    The vector bundle on the right in the definition of $\PP\Lambda_\Nd$ is of rank $g$, thus all components of $\PP\Lambda_\Nd$ are of dimension at least $g$. Let $[s]\in |\omega_C|$, and let $s_i=\beta_i^\ast s$. The fiber of $\PP\Lambda_\Nd$ above $[s]$ is
\[ \bigtimes_{i\,|\, s_i \neq 0 } \{D\in N_{i,d_i}\,|\, D\leq \divv s_i \} \times \bigtimes_{i\,|\, s_i=0} N_{i,d_i} \,.\]
Thus $p_2:\PP\Lambda_\Nd \to |\omega_C|$ is fibered above
\[ \bigcup_i \PP\left(\bigoplus_{j\neq i} \HH^0(N_j,\omega_{N_j})\right) \subset |\omega_C| \]
The fiber above this locus is of dimension $g-1$, thus every irreducible component surjects onto $|\omega_C|$ and is of dimension $g$. A general divisor in $|\omega_C|$ is non-singular. The fiber above a non-singular $D\in \Nd$ is $\PP\HH^0(C,\omega_C(-D))$ is of dimension $r(D)=\hh^0(C,D)-1$. Thus by \ref{Prop: Martens Thoerem for singular bielliptic curves}, every irreducible component of $\PP\Lambda_\Nd$ surjects onto $\Nd$. But a general point in $\Nd$ has a unique preimage, thus $\PP\Lambda_\Nd$ is birational to $\Nd$.
\end{proof}
Let $b':\PP\Lambda_{\tNd}\to \PP\Lambda_\Nd$ be the strict transform of $\PP\Lambda_\Nd$ along the blowup $\tNd\times |\omega_C|\to \Nd\times |\omega_C|$. We have the following commutative diagram
\begin{equation}\label{Tikzcd: diagram defining Lambda tNd}
    \begin{tikzcd}
        \tNd \arrow[d,"b"'] & \PP\Lambda_\tNd \arrow[l,"\tilde{p}"] \arrow[rr,"\gamma_\tNd",bend left=15] \arrow[d,"b'"] \arrow[r,phantom,"\subset"] & {\tNd\times|\omega_C|} \arrow[d,"b\times Id"] \arrow[r ]& {|\omega_C|} \arrow[d,phantom, "=" rotate=90] \\
        \Nd & \PP\Lambda_\Nd \arrow[l,"p"] \arrow[rr,"\gamma_{\Nd}"',bend right=15] \arrow[r,phantom, "\subset"] & {\Nd\times |\omega_C|} \arrow[r] & {|\omega_C|}
    \end{tikzcd}
\end{equation}
We have the following:
\begin{proposition}\label{Prop: positive-dimensional fibers of b'}
    The locus above which the fibers of $b'$ are positive-dimensional is the set $(D,H)\in \Lambda_\Nd$ such that $P_i^0+P_j^\infty \leq D\leq \beta^\ast H$ and $P_i^0+P_j^\infty \leq \beta^\ast H-D$ for some $1\leq i,j\leq n$.
\end{proposition}
\begin{proof}
Recall that $b'$ is the blowup of $B'\coloneqq \PP\Lambda_{\Nd}\cap(B\times |\omega_C|)$ where 
\[ B=\{D\in \Nd\,|\, P_i^0+P_j^\infty \leq D \,,\, \text{for some $1\leq i,j\leq n$}\}\,.\]
Let $(D,H)\in B'$, then $H$ must be singular and by \ref{Lem: H in omegaC is singular iff H in omegaN} we have $H\in |\omega_C|^-$. Let $\tilde{H}=\rho(H)\in |\omega_N|$. Suppose first that for all $i$ such that $P_i^0\leq D$, the multiplicity of $P_i^0$ in $D$ and $\beta^\ast H$ is the same. Fix $i_0,j_0$ such that $P_{i_0}^0+P_{j_0}^\infty \leq D$. Let $X_{P_{j_0}^\infty}\subset \Nd$ be the set of divisors containing $P_{j_0}^\infty$. Then locally near $(D,H)$ we have
\[ B'=\Lambda_\Nd \cap (X_{P_{j_0}^\infty} \times |\omega_C|)\,.\]
Indeed, locally near $(D,H)$ we have $\Lambda_\Nd \cap (X_{P_{j_0}^\infty} \times |\omega_C|) \subset \Nd\times |\omega_C|^-$ thus for any $(D',H')\in \Lambda_\Nd \cap (X_{P_{j_0}^\infty} \times |\omega_C|) $ near $(D,H)$, we have
\[ P_i^0\leq H' \,, \] 
$D'$ must contain $P_{i_0}^0$ and thus $(D',H')\in B'$. Thus $B'$ is locally a Cartier divisor and $b'$ is a local isomorphism. The same reasoning applies if for all $1\leq j \leq n$, the multiplicity of $P_j^\infty$ in $D$ and $\beta^\ast H$ is the same. \\
Conversely, assume that $P_i^0+P_j^\infty \leq D\leq \beta^\ast H$ and $P_i^0+P_j^\infty \leq \beta^\ast H-D$ for some $1\leq i,j\leq n$. Let $a^0$ (resp. $a^\infty$) be the multiplicity of $P_i^0$ (resp. $P_i^\infty$) in $D$. For any local parametrization $P_i^0(t),P_j^\infty(t)$ we can find a parametrization $H(t)\in |\omega_C|^-$ such that $a^0 P_i^0(t)+b^0P_j^\infty(t)\leq \beta^\ast H(t)$, and thus a family $(D(t),H(t))\in \Lambda_\Nd$ such that $a^0P_i^0(t)+a^\infty P_j^\infty \leq D(t)$. Thus the strict transform $\Lambda_\tNd$ contains the whole fiber of the blowup $b$ at $D\in \Nd$.
\end{proof}
We then have:
\begin{proposition}\label{Prop: Fibers of Lambda Nd to omega N} The projection
\[\gamma_{\tNd}: \PP\Lambda_\tNd \to |\omega_C| \]
is finite above $|\omega_C|\setminus |\omega_C|^-$. Let $H\in |\omega_C|^-$, assume $\rho(H)=\divv s$ with $s=s_1+\cdots+s_n\in \oplus_i\HH^0(N_i,\omega_{N_i})$. The fiber above $H$ is positive-dimensional in only the two following cases:
\begin{enumerate}
    \item $s_i=0$ for some $1\leq i\leq n$. Then the fiber is
    \[
    (b')^{-1}\left( \prod_{i\,|\, s_i\neq 0} \{ D\in N_{i,d_i}\,|\, D\leq P_i^0+P_i^\infty+\divv s_i \} \times \prod_{i\,, s_i=0} N_{i,d_i}\times\{H\}\right)\,. \]
    \item $P_i^0+P_j^\infty\leq \divv s$ for some $1\leq i,j\leq n$. For all such $i,j$, and for all $D\in \Nd$ such that
    \begin{align*}
        P_i^0+P_j^\infty \leq D\leq \beta^\ast H -P_i^0-P_j^\infty \,,
    \end{align*}
    $\PPN\restr{D}\times\{H\}\subset \Lambda_\tNd$ is in the fiber above $H$.
\end{enumerate}
\end{proposition}
\begin{proof}
The projection decomposes as
\[ \PP\Lambda_\tNd \overset{b'}{\longrightarrow}\, \PP\Lambda_\Nd \overset{\gamma_\Nd}{\longrightarrow} |\omega_C| \,. \] 
The first case are the positive-dimensional fibers of $\gamma_\Nd$ and follows from the proof of \ref{Prop: Lambda Nd is irreducible of dim g}. The second case corresponds to the positive-dimensional fibers of $b'$ and follows from \ref{Prop: positive-dimensional fibers of b'}.
\end{proof}
Consider the inclusion $\Pic^\ud(C)\times |\omega_C|\subset \PPP\times |\omega_C|$. Let
\[ {\PP\Lambda_{\overline{\Theta}}}\subset \PPP\times |\omega_C| \]
denote the closure of $\PP\Lambda_\Theta$.
\begin{proposition}\label{Prop: q ast Lambda theta = Lambda Nd}
 With the above notations, we have
   \[ {\PP\Lambda_{\overline{\Theta}}}=(\alpha\times Id)_\ast(\PP\Lambda_\tNd)\,.\]

\end{proposition}
\begin{proof}
Both are reduced, irreducible and agree on an open dense subset.
\end{proof}

We have the following:
\begin{proposition}\label{Prop: description of tNd}
    Suppose $\ud=(g)$ or $\ud=(1,g-1)$. Then above the locus of line bundles $\Pic^\ud(C) \subset \PPP$, $\tNd$ parameterizes line bundles together with a ``divisor"
    \[ \tNd\restr{\Pic^\ud(C)}\simeq \{(L,[s])\,|\, L\in \Pic^\ud(C)\,,  [s]\in \PP \HH^0(C,L) \} \,. \]
\end{proposition}
\begin{proof}
    Recall from \ref{Diag: Def of tNd with more details} the following commutative diagram
    \begin{center} \begin{tikzcd}
    \tilde{N}_{\underline{d}} \arrow[dr ,"b"'] \arrow[rr, bend left=20, "\alpha"] \arrow[r,hook,"i_\tNd"'] 
         & \PPN  \arrow[r] \arrow[d,"q_N"] & \PPP \arrow[d,"q"]\\  
         & N_{\underline{d}}  \arrow[r,"\alpha_N"'] & \Pic^\ud(N)
    \end{tikzcd}\,. 
    \end{center}
    Given a point in $ x\in \tNd\restr{\Pic^\ud(C) }$, we thus have a line bundle $L_x\coloneqq \alpha(x)\in \Pic^\ud(C)$ and a divisor $D_x\coloneqq b(x)\in N_\ud$. If $D_x$ is non-singular it corresponds immediately to a unique Cartier divisor. We now assume $D_x$ to be singular. Suppose first that $\ud=(g)$. A Cartier divisor on $C$ is given by
    \[D=(\lambda,a,b)_Q+D'\]
    where $D'$ is a non-singular divisor on $C$ and $(\lambda,a,b)_Q\in \CC^\ast\times \ZZ\times \ZZ$ is a Cartier divisor supported on the unique singular point $Q\in C$. $a,b$ and $D'$ are determined uniquely by $D_x$ and for a given $a,b$ and $D'$ there is a unique $\lambda\in \CC^\ast$ such that $\BO_C(D)=L_x$. \\
    Suppose $\ud=(1,g-1)$. We have $D_x=(D_1,D_2)\in N_1\times N_{2,g-1}$. Suppose first that 
    \[ P_2^0+P_2^\infty \leq D_2 \,. \]
    Since $D_1$ is of degree $1$, it can't contain both $P_1^0$ and $P_1^\infty$. Thus any section of $L_x$ vanishing at $D_2$ must vanish on $N_1$. Thus up to scalar, there is a unique section $s\in\HH^0(C,L_x)$ vanishing at $D_2\cup N_1$. \\
    We now assume $P_2^0+P_2^\infty\nleq D_2$. Assume for instance $D_2=a\cdot P_2^0+D'_2$ with $D'_2$ non-singular. By assumption $\alpha(x)=L_x$ is a line bundle. This implies $D_1=P_1^\infty$ (this comes from the description of $\tNd$ as a blow-up). For the same reason as in the irreducible case, there is now a unique $\lambda\in \CC^\ast$ such that
    \[ D=(\lambda,1,a)_{Q_1}+D'_2\]
    corresponds to $L_x$, where $Q_1$ is the singular point corresponding to $P_2^0$ and $P_1^\infty$. \par 
    Finally given $L\in \Pic^\ud(C)$ and $D\in \PP\HH^0(C,L)$, then $(L,\beta^\ast D)\in \PPN$ is in $\tNd$ and this gives the inverse of the map constructed above.
\end{proof}
\begin{proposition}\label{Prop: description of Lambda tNd}
    Suppose $\ud=(g)$ or $\ud=(1,g-1)$. Let 
    \[\PP\Lambda_\tNd^\ast \coloneqq \left\{ (L,[s_1],[s_2])\,\Big|\,(L,[s_1])\in \tNd\restr{\Pic^\ud(C)}\,, [s_2]\in \PP \HH^0(C,\omega_C\otimes L^{-1}) \right\}\,. \]
    The map
    \begin{align*}
        \PP\Lambda_\tNd^\ast &\hookrightarrow \tNd \times |\omega_C| \\
        (L,[s_1],[s_2]) &\mapsto (L,[s_1]),[s_1\otimes s_2]
    \end{align*}
    identifies $\PP\Lambda_\tNd^\ast$ with $\PP\Lambda_\tNd\restr{\Pic^\ud(C)}$.
\end{proposition}
\begin{proof}
    By \ref{Prop: Martens Thoerem for singular bielliptic curves}, the projection $\PP\Lambda_\tNd^\ast\to \tNd$ is birational. In particular, $\PP\Lambda_\tNd^\ast$ is irreducible. From the case of smooth curves we know that $\PP\Lambda_\tNd\restr{\Pic^\ud(C)}$ and $\PP\Lambda_\tNd^\ast$ coincide over the open locus of non-singular divisors \cite[246]{arbarello}. Since both are irreducible, they are equal.
\end{proof}
\begin{corollary}\label{Cor: fibers of Lambda tNd above s+0 cannot be line bundles}
    Suppose $\ud=(1,g-1)$, let $M=[s]\in |\omega_C|^-$ such that $ s\restr{N_2}=0$. Then the fiber of $\PP\Lambda_\tNd \to |\omega_C|$ above $M$ is supported above $\PPP\setminus \Pic^\ud(C)$.
\end{corollary}
\begin{proof}
    Suppose the contrary. By \ref{Prop: description of Lambda tNd} there is $(L,[s_1],[s_2])\in \PP\Lambda_{\tNd}^\ast$ such that 
    \[s_1 \otimes s_2=s \] 
    vanishes on $N_2$. But neither $s_1$ nor $s_2$ can vanish on all of $N_2$: Since the degree of the restriction of $s_1$ and $s_2$ to $N_1$ is $1$, if they vanish at both $P_1^0$ and $P_1^\infty$ they would be zero on $N_1$ as well.
\end{proof}
We end this section by introducing the following involution on $\PP\Lambda_\tNd^\ast$:
\begin{equation}\label{Definition of omega Lambda}
\begin{split}
    \omega_\Lambda : \PP\Lambda_\tNd^\ast & \to \PP\Lambda_\tNd^\ast  \\
    (L,[s_1],[s_2]) & \mapsto (\omega_C\otimes L^{-1},[s_2],[s_1])   \,,
\end{split}
\end{equation}
and 
\begin{equation}\label{Equ: Definition of tau lambda}
    \begin{split}
         \tau_\Lambda: \PP\Lambda_\tNd^\ast &\to \PP\Lambda_\tNd^\ast \\
         (L,[s_1],[s_2]) &\mapsto (\tau^\ast L, [\tau^\ast s_1], [\tau^\ast s_2] )\,.
    \end{split}
\end{equation}
By abuse of notation denote by $\tau$ the involution induced by $\tau$ on $|\omega_C|$. Clearly, we have

\begin{equation}\label{Equ: identities verified by omega Lambda}
\begin{aligned} 
\gamma_\tNd \circ \omega_\Lambda &= \gamma_\tNd \,,\\
\Nm \circ \omega_\Lambda(-) &= 2\cdot \delta - \Nm(-)\,, 
\end{aligned} \qquad 
\begin{aligned}
 \gamma_\tNd \circ \tau_\Lambda &= \tau \circ \gamma_\tNd \,,\\
 \Nm \circ \tau_\Lambda &= \Nm \,.
\end{aligned}
\end{equation}

\subsection{Chern-Mather class of the theta divisor}
We now prove the following:
\begin{lemma}\label{Lem: BECurves: Chern-Mather of theta}
Suppose $\underline{d}=(g)$ or $\ud=(1,g-1)$, then 
\[ \left[\PP\Lambda_\Theta\right] = \sum_{r=0}^g h^{r+1} \frac{\theta^{g-r}}{(g-r)!} \binom{2g-2r-2}{g-r-1}\cap \left[\PP T^\vee JC\right] \in \HH_{2g}(\PP T^\vee JC,\QQ) \,, \]
where $\theta$ is the pullback of the polarization on $JN$ and $h$ is the hyperplane class in $\PP T^\vee_0 JC$.
\end{lemma}
\begin{remark}
    Our proof gives a recipe to do the above computation for a general $\ud$, but as the computation would become much more cumbersome, we restrict to these cases. We expect the formula to be more complicated in the general case.
\end{remark}
\begin{proof}
\emph{Case $\ud=(g)$.} Let $x=[N_{g-1}]\in \HH^2(N_{g},\QQ)$ and $\theta \in \HH^2(JN,\QQ)$ denote the class of the polarization, $h\in \HH^2(|\omega_C|,\QQ)$ and $h'\in \HH^2(\PPP,\QQ)$ denote the respective hyperplane classes. By \ref{Equ: Class of tilde Nd (blowup of Nd)}, \ref{Equ: chern class of E K,i (singular cas)}, \ref{Cor: Class of Lambda Nd} and \ref{Tikzcd: diagram defining Lambda tNd} we have in $\HH_{2g}(\PPN \times |\omega_C| ,\QQ)$
\begin{align*} 
\left[(b\times\Id)^\ast \PP\Lambda_{\Nd}\right]&=\left(\sum_{r=0}^g h^r c_{g-r}( E_{K}) \right)\cap \left[\tNd\times |\omega_C|\right]\\
&=(x+h')\left(\sum_{r=0}^g h^r c_{g-r}( E_{K}) \right) \cap \left[\PPN\times |\omega_C|\right] \,.
\end{align*}
Recall from \ref{Sec: The Abel-Jacobi Map} that the center of the blowup $b:\tNd\to \Nd$ is
\[ B= \{ P^0+P^\infty +N_{g-2} \}\subset \Nd \,. \]
Let $B'=(\PP\Lambda_\Nd \cap (B\times |\omega_C|)) $, then
\begin{align*}
    B'=\{ (D,H)\in N_{g-2}\times |\omega_C|\,\big|\, P^0+P^\infty + D \leq \beta^\ast H \}
\end{align*}
A canonical divisor containing $P^0$ must be in $|\omega_C|^-\subset |\omega_C|$. Thus under the identification $\rho:|\omega_C|^-\simeq  |\omega_N|$ the above is equal to the vanishing locus of the composition of maps of vector bundles on $N_{g-2}\times |\omega_N|$
\[ \BO_{|\omega_N|}(-1) \to \HH^0(N,\omega_N) \to E_{K,N_{g-2}}\,,\]
where $E_{K,N_{g-2}}$ is the corresponding evaluation bundle on $N_{g-2}$. By \cite[340]{arbarello} we have
\begin{align*} 
c_r(E_{K,N_{g-2}} )\cap [N_{g-2}] &= \sum_{k=0}^r \binom{r}{k} x ^k \frac{\theta^{r-k}}{(r-k)!} \cap [N_{g-2}] \\
&= x^2 c_r(E_{K}) \cap [N_{g}] \,.
\end{align*}
Thus
\begin{align*}
    [B']&= \sum_r h^r c_{g-r-2}(E_K) \cap [N_{g-2}\times |\omega_N|] \\
    &=x^2 h\sum_r h^r c_{g-r-2}(E_K) \cap [N_\ud \times |\omega_C|] \\
    &= x^2 \sum_r h^r c_{g-r-1}(E_K)\cap [N_\ud\times |\omega_C|] \in \HH_{2g-2}(\Nd\times |\omega_C|,\QQ)\,.
\end{align*}
By the blowup formula \cite[Th. 6.7]{Fulton1998} and \ref{Equ: alpha pushforward classes of E K,i (singular case)} we have
\begin{align*}
 (\alpha\times \Id)_\ast [\PP\Lambda_\tNd]&= (\alpha\times \Id)_\ast \left( (b\times \Id)^\ast [\PP\Lambda_\Nd]- (q_N\times\Id)^\ast [ B' ] \right)  \\
  &=\sum_r h^r \left( h' \frac{\theta^{g-r}}{(g-r)!} \binom{2g-2r}{g-r} +\frac{\theta^{g-r+1}}{(g-r+1)!} \binom{2g-2r+1}{g-r+1}\right.\\
  &\qquad \left.-\frac{\theta^{g-r+1}}{(g-r+1)!} \binom{2g-2r}{g-r+1} \right) \\
  &= \sum_r h^r h' \frac{\theta^{g-r}}{(g-r)!} \binom{2g-2r}{g-r} +  h^r \frac{\theta^{g-r+1}}{(g-r+1)!} \binom{2g-2r}{g-r} \\
  &= \sum_r h^r \frac{(\theta + h')^{g-r+1}}{(g-r+1)!} \binom{2g-2r}{g-r}\cap \left[ \PPP\times |\omega_C| \right] \,. 
   \end{align*}

\emph{Case $\underline{d}=(1.g-1)$.} For $i=1,2$ let $x_i=[N_{i,d_i-1}]\in \HH^2(N_{i,d_i},\QQ)$ and $\theta_i \in \HH^2(JN_i,\QQ)$ denote the class of the polarization. By \ref{Equ: Class of tilde Nd (blowup of Nd)}, \ref{Equ: chern class of E K,i (singular cas)} and \ref{Cor: Class of Lambda Nd} we have
\begin{align*} 
[(b\times \Id)^\ast  \PP\Lambda_{\Nd}]&=(x_1+x_2+h')\left(\sum_{r=0}^g h^r c_{g-r}( E_{K,1}\oplus E_{K,2}) \right) \\
&= (x_1+x_2+h')\left(\sum_{r=0}^g h^r 2 x_1 c_{g-r-1} + c_{g-r}\right)\\
&\in \HH_{2g}(\PP^1\times \Nd \times |\omega_C|,\QQ ) \,, 
\end{align*}
where we denote $c_r(E_{K,2})$ by $c_r$. Recall from \ref{Sec: The Abel-Jacobi Map} that the center of the blowup $\tNd\to \Nd$ is
\[ B=  B_{12}\cup B_{21} \cup B_{22} \subset \Nd\,, \]
with 
\begin{align*}
    B_{12}&=\{P_1^0\}\times \{ P_2^\infty + N_{2,d_2-1} \}\,, \\
    B_{21}&=\{P_1^\infty\}\times \{ P_2^0+ N_{2,d_2-1} \} \,, \\
    B_{22}&= N_1 \times \{ P_2^0+P_2^\infty +N_{2,d_2-2} \} \,.
\end{align*}
Let $B'_{ij}=(\PP\Lambda_\Nd \cap (B_{ij}\times |\omega_C|))$, then
\begin{align*}
    B'_{12}=\{ D,H\in N_{2,d_2-1}\times |\omega_C|\,\big|\, P_1^0+P_2^\infty + D \leq \beta^\ast H \}
\end{align*}
Since a canonical divisor containing $P_1^0$ must be in $|\omega_N|\simeq |\omega_C|^-\subset |\omega_C|$, the above is equal to the vanishing locus of the composition of maps of vector bundles on $N_{2,d_2-1}\times |\omega_N|$
\[ \BO_{|\omega_N|}(-1) \to \HH^0(N,\omega_N) \to E_{\omega_{N_2}(P_2^0)}\,,\]
where $E_{\omega_{N_2}(P_2^0)}$ is the corresponding evaluation bundle on $N_{2,d_2-1}$. Notice that
\begin{align*} 
c_r(E_{\omega_{N_2}(P_2^0)} )\cap N_{2,d_2-1} &= \sum_{k=0}^r \binom{r}{k} x_2 ^k \frac{\theta_2^{r-k}}{(r-k)!} \cap [N_{2,d_2-1}] \\
&= x_2 c_r(E_{K,2}) \cap [N_{2,d_2}] \,
\end{align*}
thus
\begin{align*}
    [B'_{12}]&= \sum_r h^r c_{g-r-2}(E_{\omega_{N_2}(P_2^0)}) \cap [N_{2,d_2-1}\times |\omega_N|] \\
    &= x_1 x_2  \sum_{r} h^{r+1} c_{g-r-2} \cap [\Nd\times |\omega_C|]\in \HH_{2g-2}(\Nd\times |\omega_C|,\QQ)\,.
\end{align*}
Clearly
\[ [B'_{12}]=[B'_{21}] \,. \]
In the same way, we have that $B'_{22}$ is $N_1$ times the vanishing locus of the composition of morphism of vector bundles on $N_{2,d_2-2}$
\[ \BO_{|\omega_{N_2}|}(-1) \to \HH^0(N_2,\omega_{N_2}) \to E_{\omega_{N_2}} \,.\]
Again $c_r(E_{\omega_{N_2}})\cap [N_{2,d_2-2}]=x_2^2 c_r \cap [N_{2,d_2}]$, thus
\begin{align*}
    [B'_{22}]&= \sum_r h^r c_{g-r-3}(E_{\omega_{N_2}}) \cap [N_1\times N_{2,d_2-2}\times |\omega_{N_2}|] \\
    &= x_2^2 \sum_{r} h^{r+2} c_{g-r-3} \cap [\Nd\times |\omega_C|]\in \HH_{2g-2}(\Nd\times |\omega_C|,\QQ)\,.
\end{align*}
By the blowup formula \cite[Th. 6.7]{Fulton1998}, we have in $\HH_{2g}(\PPN \times |\omega_C|,\QQ )$
\begin{align*}
    [\PP\Lambda_{\tNd}]&=(b\times \Id)^\ast [\PP\Lambda_\Nd]- (q_N\times \Id)^\ast[B'] \\
    &=(b\times\Id)^\ast [\PP\Lambda_\Nd]- (q_N\times\Id)^\ast\left(2[B'_{12}]-[B'_{22}]\right) \\
    &=(x_1+x_2+h')\left(\sum_{r=0}^g h^r (2x_1 c_{g-r-1} + c_{g-r})\right) - 2x_1x_2\sum_r h^r c_{g-r-1} - x_2^2\sum_{r} h^r c_{g-r-1}  \\
    &=\sum_r h^r \left( h'(2x_1 c_{g-r-1}+c_{g-r}) + (x_1+x_2) c_{g-r}-x_2^2 c_{g-r-1})\right) \cap [\PPN \times |\omega_C|]\,.
\end{align*}
Thus by \ref{Equ: alpha pushforward classes of E K,i (singular case)} we have in $\HH_{2g}(\PPP\times|\omega_C|,\QQ)$
\begin{align*}
  (\alpha\times \Id)_\ast[\PP\Lambda_\tNd]&= \sum_r h^r \left( h'(2 \theta_1 \frac{\theta_2^{g-r-1}}{(g-r-1)!} \binom{2g-2r-2}{g-r-1}+ \frac{\theta_2^{g-r}}{(g-r)!} \binom{2g-2r}{g-r}) \right.\\
    & \quad \left. +\theta_1 \frac{\theta_2^{g-r}}{(g-r)!} \binom{2g-2r}{g-r}+\frac{\theta_2^{g-r+1}}{(g-r+1)!} \binom{2g-2r+1}{g-r+1} - \frac{\theta_2^{g-r+1}}{(g-r+1)!} \binom{2g-2r}{g-r+1} \right)\\
   &=h'\sum_r h^r\left( 2 \theta_1 \frac{\theta_2^{g-r-1}}{(g-r-1)!} \binom{2g-2r-2}{g-r-1}+ \frac{\theta_2^{g-r}}{(g-r)!} \binom{2g-2r}{g-r}\right) \\
   &\quad + \sum_r h^r \frac{(\theta_1+\theta_2)^{g-r+1}}{(g-r+1)!}\binom{2g-2r}{g-r}  \\
   &=\left(\sum_r h^r \frac{(\theta_1+\theta_2+h')^{g-r+1}}{(g-r+1)!}\binom{2g-2r}{g-r}-h^rh' \theta_1 \frac{\theta_2^{g-r}}{(g-r)!}\binom{2g-2r-2}{g-r}\right) \,. 
\end{align*}
The Lemma then follows from \ref{Prop: q ast Lambda theta = Lambda Nd}, and the fact that $[h'\restr{JC}]=0$.
\end{proof}

\section{Pryms associated to biellitpic curves}\label{Sec: Pryms associated to biellitpic curves}
We keep the notations of Section \ref{Sec: Singular bielliptic curves}, i.e. $C$ is a nodal curve of genus $g+1$, $\pi:C\to E$ is a double covering of type (3), $E$ is a cycle of $n$ $\PP^1$'s, $\Delta$ is the branch locus of $\pi$ and $\ud\coloneqq \underline{\deg}(\Delta)/2$. Moreover we now fix $\delta\in \Pic^\ud(E)$ with $\Delta\in |\delta^{\otimes 2}|$. We define
\begin{align*} 
P &\coloneqq \{L\in \Pic^\ud(C)\,|\, \Nm(L)=\delta \} \subset \Pic^\ud(C) \,, \\
\Xi  &\coloneqq \Theta\cap P \subset P \,.
\end{align*}
These notations are fixed for the remainder of Section \ref{Sec: Pryms associated to biellitpic curves}. We have the following commutative diagram, whose rows and columns are exact \cite{Beauville1977}
\[
\begin{tikzcd}
  & 0 \arrow[d] & 0 \arrow[d] & 0 \arrow[d] & \\
    0 \arrow[r] &\ZZ/2\ZZ \arrow[d,hook] \arrow[r] & P \arrow[d,hook] \arrow[r] & \Pic^\ud(N) \arrow[d] \arrow[r] & 0 \\
    0 \arrow[r] & \CC^\ast \arrow[r] \arrow[d, "z \mapsto z^2" lablrot] \arrow[d] & \Pic^\ud(C) \arrow[r,"\beta^\ast"] \arrow[d,"\Nm"] & \Pic^\ud(N) \arrow[d] \arrow[r] & 0 \\
   0 \arrow[r] & \CC^\ast \arrow[r] \arrow[d] & \Pic^\ud(E) \arrow[d] \arrow[r] & 0 & \\
   & 0 & 0 & &
\end{tikzcd}\,.
\]
In particular, there is a degree $2$ isogeny of polarized abelian varieties $P\to \Pic^\ud(N)$. We thus have an identification
\begin{equation}\label{Equ: Tvee P and Tvee JN is omegaN} 
T^\vee_0 P \simeq T^\vee_0 JN = \HH^0(N,\omega_N)\,. 
\end{equation}
Let
\[W\coloneqq \overline{\Nd \times_{\Pic^\ud(E)} \{\delta\} }\,, \quad \text{and} \quad \tilde{W}\coloneqq \tNd\times_{\Pic^\ud(E)}\{\delta\} \,. \]
Let $R$ be the ramification divisor of $\pi: C\to E$, and
\[ \Wasing  \coloneqq \{D\in W\,|\, D\leq R \} \,, \qquad  \Wtasing \coloneqq b^{-1}(\Wasing)\subset \tilde{W} \,, \] 
Recall that $R$ is non-singular thus $b$ is a local isomorphism near $\Wasing$. The Abel-Jacobi map $\alpha$ restricts to a map $\phi\coloneqq \alpha\restr{\tilde{W}}: \tilde{W}\to \Xi$. We have the following: 
\begin{lemma}
    The singular locus of $\tilde{W}$ is
    \[ \Sing(\tilde{W})=\left(b^{-1}(B^{0\infty})\cap \tilde{W}  \right)\cup \Wtasing \,,\]
    where
    \begin{align*} 
    B^{0\infty}&\coloneqq \{ D\in \Nd\,|\, P_i^0+P_j^0+P_k^\infty+P_l^\infty \leq D \,, \text{ for some $i\neq j$ and $k\neq l$.} \} \,.
    \end{align*}
    Moreover, at a point $x\in \Wtasing$, $\tilde{W}$ has a quadratic singularity of maximal rank, i.e. locally analytically we have
    \[ (\tilde{W},x) \simeq (V(x_1^2+\cdots+x_g^2),0) \subset (\AAA^g,0) \,. \]
\end{lemma}
\begin{proof}
    Let $\tilde{D}\in \tilde{W}$, and $D=b(\tilde{D})$. Recall that $b:\tNd \to \Nd$ is the blowup at $B\coloneqq \{ D\in \Nd\,|\, P_i^0+P_j^\infty \leq D \}$. \par 
    \emph{Step 1: Suppose $D\notin B$.} Then $b$ is a local isomorphism at $D$, and thus induces a local isomorphism $(\tilde{W},\tilde{D})\to (W,D)$. Suppose that $D\leq R$. We can assume that $D=P_1+\cdots+P_g$. For $1\leq i \leq g$, the morphism $\pi_N:N\to E$ is ramified at $P_i$ thus there are local coordinates $z_i$ on $N$ centered at $P_i$ such that
    \[ \pi_N(z_i)=z_i^2+Q_i \,, \quad \text{where $Q_i=\pi(P_i)$}\,. \]
    Moreover $z_1,\dots,z_g$ define coordinates on $\Nd$ locally near $D$. On $JE=\CC^\ast$ the group law is multiplicative thus the condition to map to $\delta$ by $\Nm$ reduces locally near $D$ to 
    \[ \prod_{i=1}^g (z_i^2+Q_i) - Q_1\cdots Q_g =0 \,, \]
    where we view the points $\alpha_N(Q_i)\in \Pic^1(E)\simeq \CC^\ast$ as complex numbers by abuse of notations. The Hessian of the above function is non-degenerate, thus by the Morse Lemma this is a quadratic singularity of maximal rank. \par  
    We now assume $D\nleq R$. Then $D=D'+P_0$ for some point $P_0\nleq R$. Locally near $P_0$ there is the embedding
    \begin{align*}
       i_{D'} : N &\hookrightarrow \Nd \\
        P&\mapsto P+D' \,.
    \end{align*}
    The composite $\Nm\circ i_{D'}$ has non-zero differential at $P_0$ thus $\Nm:\Nd \dashrightarrow \PP^1$ has non-zero differential at $D$. Thus $W$ (resp. $\tilde{W}$) is smooth at $D$ (resp. $\tilde{D}$). \par 
    \emph{Step 2: Suppose $D\in B$.} Let $U\subset \Nd$ be an open set and $(q_N\restr{U},\lambda,\mu):\PPN\restr{U}\to U\times \PP^1$ be a local trivialization of $\PPN$ such that $\tNd$ is the vanishing locus of
    \[ \lambda s_1^0 s_2^0\cdots s_n^0 - \mu s_1^\infty s_2^\infty \cdots s_n^\infty \]
    as in \ref{Equ: equation defining tNd}, where for $k\in \{0,\infty\}$, $1\leq i \leq n$,
    \[ \divv s_i^k= B_i^k= P_i^k+ N_{\ud-e_i} \subset N_\ud\,. \]
    Above the trivialization $U$, the norm map becomes
    \begin{align*}
        \Nm: U\times \PP^1 &\to \PP^1 \\
        (D,\lambda:\mu) &\mapsto (\lambda^2:\mu^2) \,.
    \end{align*}
    Under this identification we have $\delta\in \PP^1\setminus \{0,\infty\}$. The divisors $B_i^k$ are normal crossing divisors, thus the result follows.
\end{proof}
\begin{corollary}\label{Cor: tilde W is smooth}
    If $\ud=(g)$ or $\ud=(1,g-1)$, then $\Sing(\tilde{W})=\Wtasing$.
\end{corollary}
\begin{proof}
    In these two cases the set $B^{0\infty}$ is empty for degree reasons.
\end{proof}
\begin{corollary}\label{Cor: points in Xiasing a isolated singularities of maximal rank}
    Let 
    \[ \Xiasing\coloneqq \phi\left(\Wtasing\right) \subset \Xi\,. \] 
    The points of $\Xiasing$ are isolated singularities of maximal rank of $\Xi$. These correspond to the additional isolated singularities of \cite{podelski2023GaussEgt} (hence the notation).
\end{corollary}
\begin{proof}
     For a line bundle $L\in \Xiasing$, we have $\hh^0(N_i,L\restr{N_i})=1$ for $1\leq i \leq n$ thus $\hh^0(C,L)=1$ by the proof of \ref{Prop: Martens Thoerem for singular bielliptic curves}. Thus $\phi:\tilde{W}\to \Xi$ is a local isomorphism near $L$ by \ref{Prop: description of tNd} and the result follows from the lemma above.
\end{proof}

\subsection{Chern-Mather class of the Prym theta divisor}
We keep the notations of the previous section. Let
\[ \Lambda_\Xi \subset T^\vee P =P\times \HH^0(N,\omega_N) \]
be the conormal variety to $\Xi$, and $\PP\Lambda_\Xi \subset P\times |\omega_N|$ the projectivization. Consider the following composite
\[ \mathscr{F}: |\omega_C| \dashrightarrow |\omega_C|^- \overset{\rho}{\longrightarrow} |\omega_N| \,, \]
where $|\omega_C|\dashrightarrow |\omega_C|^-$ is the projection from $R\in |\omega_C|$. We have the following
\begin{proposition}\label{Prop: Lambda Xi is pushforward of Lambda theta}
With the above notations, we have
\[ \PP\Lambda_\Xi=(\Id \times \mathscr{F})_\ast\left( \PP \Lambda_\Theta\restr{P} \right) \,.\]
\end{proposition}
\begin{proof}
    Recall that we have a canonical identification $\PP T^\vee JC = JC\times |\omega_C|$. It follows from \cite[Fig. 4.24]{podelski2023GaussEgt} that for a smooth point $x\in \Xi$ we have
    \[ \GG_\Xi(x)=\mathscr{F}\circ \GG_\Theta(x)\,, \]
    where $\GG_\Xi:\Xi \dashrightarrow |\omega_N|$ and $\GG_\Theta:\Theta \dashrightarrow |\omega_C|$ are the respective Gauss maps. The proposition follows since a general point in $\Lambda_\Theta\restr{\delta}$ lies above a smooth point of $\Xi$, and $\Lambda_\Xi$ is irreducible.
\end{proof}

\begin{theorem}\label{Thm: Prym: Chern-Mather of Prym theta}
Suppose $\ud=(g)$ or $\ud=(1,g-1)$, then
\[ [\PP\Lambda_\Xi]= \sum_{r=0}^{g-1} h^r \frac{\xi^{g-r} }{(g-r)!}\binom{2g-2r-2}{g-r-1}\cap [T^\vee P]\in \HH_{2g}(T^\vee P,\QQ)\,, \]
where $h$ is the pullback of the hyperplane class in $\PP T^\vee_0 P$ and $\xi$ corresponds to the pullback of $\Xi$. In particular, the Chern-Mather classes of $\Lambda_\Xi$ are
\[ c_{M,r}\left(\Lambda_\Xi\right)= \frac{\xi^{g-r}}{(g-r)!} \binom{2g-2r-2}{g-r-1}\in \HH_{2r}(P,\QQ) \,. \]
\end{theorem}
\begin{proof}
Follows from \ref{Prop: Lambda Xi is pushforward of Lambda theta} and \ref{Lem: BECurves: Chern-Mather of theta}, and the fact that
\[ \theta \cap [P]=\xi \,, \quad \text{and} \quad \mathscr{F}_\ast h_C^{r+1}=h_N^r\,,\]
where $h_C$ and $h_N$ are the hyperplane classes on $|\omega_C|$ and $|\omega_N|$ respectively.
\end{proof}

\subsection{The fibers of the Gauss map}
We will now study the fibers of the Gauss map $\gamma_\Xi:\PP\Lambda_\Xi\to |\omega_N|$ in the cases $\ud=(g)$ and $\ud=(1,g-1)$. The main result is the following:
\begin{theorem}\label{Thm: bielliptic Prym: fibers of Gauss map }
    Suppose $\ud=(g)$ or $\ud=(1,g-1)$, then away from a subset $S\subset |\omega_N|$ of codimension at least $3$, $\gamma_\Xi$ is finite.
\end{theorem}
We fix the following notations
\begin{align*}
    M&=[s_1+\dots+s_n]\in |\omega_N|=\PP(\oplus_i \HH^0(N_i,\omega_{N_i}))\,, & H&=\rho^{-1}(M) \\
    V_M&=\mathscr{F}^{-1}(M)=\langle H , R \rangle \subset |\omega_C|\,, & Z_M&=\Lambda^\ast_\tNd\restr{V_M}\,.
\end{align*}
From \ref{Prop: q ast Lambda theta = Lambda Nd} and \ref{Prop: Lambda Xi is pushforward of Lambda theta} we have 
\[ \PP\Lambda_\Xi = (\alpha\times \mathscr{F})_\ast (\PP \Lambda^\ast_\tNd \cap \Nm^{-1}(\delta)) \,,\]
thus positive-dimensional fibers of $\gamma_\Xi$ above $M$ correspond to components $Z$ of $Z_M$ such that $\Nm(Z)=\delta$.

\subsubsection*{Step 1: The case of components not finite onto $V_M$.}
Suppose that there is a component $Z$ of $Z_M$ that is not finite onto $V_M$, such that $\Nm(Z)=\delta$. By \ref{Prop: Fibers of Lambda Nd to omega N} we have $\gamma_\tNd(Z)= H$. Suppose that we are in the second case of Prop. \ref{Prop: Fibers of Lambda Nd to omega N}. Then $Z\subset \PPN\restr{D}\times\{H\}\simeq \PP^1$ for some $D$. Then the norm map restricted to $\PPN\restr{D}$ is of degree $2$ thus only finitely many points lie above $\delta$, contradicting $\Nm(Z)=\delta$. \\ 
Suppose now that we are in the first case of Prop. \ref{Prop: Fibers of Lambda Nd to omega N}. Then necessarily we must be in the case $\ud=(1,g-1)$. Suppose $M=[s_1+s_2]$. We thus have either $s_1=0$ or $s_2=0$. If $s_2=0$, then by \ref{Cor: fibers of Lambda tNd above s+0 cannot be line bundles} we have $\Nm(Z)\subset\{0,\infty\}$ which contradicts $\Nm(Z)=\delta$. We now assume $s_1=0$. Let $H_2\coloneqq H\restr{N_2}=\divv s_2+P_2^0+P_2^\infty$. By \ref{Prop: Fibers of Lambda Nd to omega N}, we have
\[ b'(Z) \subset N_1\times \{ D_2\}\times \{ H\} \]
for some $D_2\leq H_2$. Suppose first that $P_2^0+P_2^\infty \nleq D_2$. Then $b':Z\to b'(Z)$ is generically finite by \ref{Prop: positive-dimensional fibers of b'} (thus finite) and for a general point $x\in Z$, we have
\[ \Nm(x)=\Nm(b'(x))\neq \delta \]
which is a contradiction. We thus have $P_2^0+P_2^\infty \leq D_2$. Consider $Y=\omega_\Lambda(Z)$. Then by \ref{Equ: identities verified by omega Lambda} we have $\gamma_\tNd(Y)=\gamma_\tNd(Z)=H$ and $\Nm(Y)=\delta$ thus $Y$ is a positive-dimensional fiber of $\gamma_\tNd$. We have
\[ b'(Y)=N_1\times \{H_2-D_2\} \,,\]
and by the above reasoning applied to $Y$ we have
\[ P_2^0+P_2^\infty \leq H_2-D_2 \,.\]
Thus we must have $M\in |\omega_{N_2}(-P_2^0-P_2^\infty)|\subset |\omega_N|$ which is of codimension $3$.

\subsubsection*{Step 2: The case of components finite onto $V_M$.}
Let $Z$ be the union of all positive-dimensional components of $\PP\Lambda_\tNd^\ast\restr{V_M}$ that are finite above $V_M$, and are mapped to $\delta$ by $\Nm$. Note that if $\ns_\delta(C/E)=0$, $Z$ is empty because by assumption no subdivisor of $R$ lies above $\delta$. This section is thus relevant only in the case $\ns_\delta(C,E)>0$. We use the notations of Fig. \ref{Tikzcd: diagram defining Lambda tNd}. Let $\pi_\ast:\Nd \to E_\ud$ be the pushforward of points. Let $Y\coloneqq \pi_\ast \circ b \circ \tilde{p}(Z)$. As a general divisor in $V_M$ is non-singular, so is a general divisor in $Y$ and we thus have en embedding
\[ j:Y\hookrightarrow |\delta|=\PP \HH^0(E,\delta) \,. \]
The involutions $\omega_\Lambda$ and $\tau_\Lambda$ from \ref{Definition of omega Lambda} and \ref{Equ: Definition of tau lambda} induce involutions on $Z$ and $Y$, which we denote by $\omega$ and $\tau$ by abuse of notation. The action of $\tau$ on $|\omega_C|$ induces an involution on $V_M$ as well which we denote by $\tau$. We thus have the following commutative diagram
    \[
    \begin{tikzcd}
       & Z \arrow[r] \arrow[d] & V_M \simeq \PP^1 \arrow[d] \\
    {|\delta |} &     Y \arrow[r] \arrow[l,hook',"j"] \arrow[d,hook,"{(j,j\circ \omega)}"] & {V_M/\tau \simeq \PP^1 }  \arrow[d,hook,"i"]  \\
     & {|\delta|\times |\delta|} \arrow[r,"m"] & {|\Delta|}
    \end{tikzcd}\,,
    \]
where $m:|\delta|\times|\delta|\to |\Delta|$ is the multiplication map. We then have the following:
\begin{lemma}\label{Lem: degree of Y compared to deg Z to V}
        \[ \deg [Y] = \deg( Z\to V_M)/2 \,,\]
    where $\deg [Y]=\deg [Y]\cap c_1(\BO_{|\delta|}(1))$.
\end{lemma}
\begin{proof}
  The morphism $Z\to Y$ and $V_M\to V_M/\tau$ are generically of degree $2$, thus
    \[ \deg(Z\to V_M)=\deg(Y\to V_M/\tau) \eqcolon k \,. \]
    By definition we have
    \[V_M=\{ \divv (\lambda s_{R}+\mu \pi ^\ast s)\,|\, (\lambda:\mu)\in \PP^1 \} \subset |\omega_C|\,,\]
    where $\divv s_R=R$ and $M=[s]$. Thus 
    \[ i_\ast V_M/\tau = \{ \divv( \lambda s_{\Delta}+\mu s^2) \,|\, (\lambda,\mu)\in \PP^1 \}\subset |\Delta| \,, \]
    where $\divv s_\Delta=\Delta$. Thus $i_\ast [V_M/\tau]$ is of degree $1$. The multiplication map $m$ is the composition of the Segre embedding and a linear projection, it is thus of bidegree $(1,1)$. If $d=\deg j_\ast[Y]$, then $(j,j\circ\omega)$ is of bidegree $(d,d)$, thus
    \[ k=2d \,. \]
\end{proof}
\begin{lemma}\label{Lem: Case Eg0: locus S such that deg Z to V <=4}
    There is a closed set $S\subset |\omega_N|$ of codimension at least $3$, such that for all $M\in |\omega_N|\setminus S$, we have 
    \[ \deg(Z\to V_M)\leq 4 \,, \]
    where $Z$ is the union of all components of $\Lambda^\ast_\tNd\restr{M}$ that are finite onto $V_M$ and mapped to $\delta$ by $\Nm$.
\end{lemma}
\begin{proof}
    For every $(L,[s_1],[s_2])\in Z\restr{H}\subset\Lambda^\ast_\tNd\restr{H}$ such that $\Nm(L)=\delta$, we have $\tau (L)=\omega(L)$ and
    \[ \Nm(\omega_C\otimes L^{-1})=\delta \,. \]
    Thus points above $H$ that map to $\delta$ come in pairs. It is not complicated to see that having $3$ such pairs above $H$ imposes a condition of codimension $3$ on $M$.
\end{proof}
We can now complete the proof of Theorem \ref{Thm: bielliptic Prym: fibers of Gauss map }. By the above lemma, away from a set $S$ of codimension at least $3$, we have $\deg(Z \to V_M)\leq 4$. By \ref{Lem: degree of Y compared to deg Z to V} we then have $\deg [Y]=\deg(Z \to V_M)/2 =2$ thus $Y\simeq \PP^1$ is a rational curve. Recall that $\omega$ and $\tau$ commute. Consider the following tower of double coverings of curves   
\[ \begin{tikzcd}
        & Z \arrow[dl,"\pi_\omega "'] \arrow[d,"\pi_\tau"] \arrow[dr,"\pi_{\omega\tau}"] & \\
        Y_\omega  \arrow[dr,"p_\omega" '] & Y \arrow[d,"p_\tau"] & Y_{\omega \tau} \arrow[dl,"p_{\omega\tau}"] \\
         & \PP^1 &
    \end{tikzcd}\,, \]
    where $Y_\omega$ (resp. $Y_{\omega \tau}$) is $Z/\omega$ (resp. $Z/\omega\tau$). Since $Y\simeq \PP^1$, the lower curve has to be $\PP^1$. The fixed points of $\omega$ correspond to theta-nulls. Moreover $\omega$ doesn't fix the points in $Z_M$ above $R\in V_M$. Thus away from a finite locus in $|\omega_N|$ we can assume that $p_\omega$ is étale. For all $L\in P$ we have $L+\tau L=\pi^\ast \delta$, thus
    \begin{align*}
        \omega_C- \tau L=L+\omega_C-\pi^\ast \delta \neq L \,.
    \end{align*}
    Thus $\omega \tau$ acts fixed point free on $Z$. By the above diagram this implies that $Y\to \PP^1$ is étale which is impossible by Riemann-Hurwitz.

\subsection{The characteristic cycle}
Let $j:\Xi_\sm \hookrightarrow \Xi$ be the embedding and $\IC_\Xi\coloneqq j_{!\ast} \CC_{\Xi_\sm}[g-1]\in \mathrm{Perv(P)}$ be the intersection complex associated to $\Xi$. We now compute the characteristic cycle $\mathrm{CC}(\IC_\Xi)$ for $\ud=(g)$ and $\ud=(1,g-1)$. The proof is inspired from Bressler and Brylinski's proof of the irreducibility of the characteristic cycle of the theta divisor of non-hyperelliptic Jacobians \cite{BresslerBrylinski97}. Recall that the restriction of the Abel-Jacobi map $\alpha:\tNd\to \overline{\Theta}$ induces a map
\[ \phi \coloneqq \alpha\restr{\tilde{W}}:\tilde{W} \to \Xi \,. \] 
Let $\tilde{W}^o\coloneqq \tilde{W}\setminus \Wtasing$, $\Xi^o\coloneqq \Xi\setminus \Xiasing$, and $\phi^o:\tilde{W}^o\to \Xi^o$ be the restriction. By \ref{Cor: tilde W is smooth}, $\tilde{W}^o$ is smooth if $\ud=(g)$ or $\ud=(1,g-1)$. Moreover a general line bundle $L\in \Xi$ verifies $\hh^0(C,L)=1$, thus $\phi$ is birational by \ref{Prop: description of tNd}.
We start with the following:
\begin{lemma}\label{Lem: Kernel of dd alpha in singular case}
    Suppose $\ud=(g)$ or $\ud=(1,g-1)$, and $D\in \tilde{W}^o$, then
    \[ \dim \Ker( {}^t \dd_D \phi ) \leq \hh^0(N,\omega_N(-\beta^\ast D))+1 \,. \]
    If moreover $D$ is non-singular, then
    \[ \Ker({}^t \dd \phi)=\left(\HH^0(C,\omega_C(-D))+\langle s_R \rangle \right)\cap \HH^0(N,\omega_N) \,. \]
\end{lemma}
\begin{proof}
    Recall the following commutative diagram defining $\phi$
    \[ \begin{tikzcd}
        \tilde{W} \arrow[rrrr,bend left=15, "\phi"] \arrow[d,"b\restr{W}"'] \arrow[r,hook] & \tNd \arrow[r,"\alpha"] \arrow[d,"b"] & \PPP \arrow[d,"\beta^\ast"] & P \arrow[l,hook'] &\Xi \arrow[l,hook'] \arrow[dl,"\beta^\ast\restr{\Xi}"] \\
        W \arrow[r,hook] \arrow[rrr,bend right=15, "\phi_N"'] & \Nd \arrow[r,"\alpha_N"] & \Pic^\ud(N) & \Xi' \arrow[l,hook']
    \end{tikzcd}\,. \]
    Let $D\in \tilde{W}$, with $D\nleq R$. $\beta^\ast\restr{P}:P\to \Pic^\ud(N)$ is a degree $2$ isogeny, so composing $\phi$ with it doesn't change the codifferential. By \cite[Lem. 2.3 p. 171]{arbarello} we have
    \[ \Ker({}^t\dd\alpha_N)=\HH^0(N,\omega_N(-\beta^\ast D)) \,. \]
    Localy near $D$, $\tilde{W}$ is smooth of codimension $1$ in $\tNd$, and $\tilde{W}\to W$ is the normalization, thus the codifferential is injective, thus
    \[ \dim \Ker({}^t\dd\phi)\leq \dim \Ker({}^t \dd \alpha_N)+1=\HH^0(N,\omega_N(-\beta^\ast D))+1 \,. \]
    \par 
    Now suppose $D$ non-singular. Then the proof of Lemma 2.3 page 171 in \cite{arbarello} can be repeated at verbatim locally near $D$ and thus
    \[ \Ker({}^t \dd_D \alpha)=\HH^0(C,\omega_C(-D)) \,. \]
    We then have the following commutative diagram 
    \[ 
    \begin{tikzcd}
     T^\vee_D \tilde{W} & T^\vee_D \tNd \arrow[l,two heads,"{}^t\dd \iota_{\tilde{W}}"] & T^\vee_{\alpha(D)} \PPP \arrow[d,equal] \arrow[l,"{{}^t\dd \alpha}"] & T^\vee_{\alpha(D)} P \arrow[l,hook'] \arrow[lll,"^t\dd \phi"',bend right=15] \arrow[d,equal]\\
     && \HH^0(C,\omega_C) & \HH^0(N,\omega_N) \arrow[l,hook']
\end{tikzcd}\,.
\]
Thus
\[ 0\neq \langle {}^t\dd \alpha(s_R) \rangle = \Ker( {}^t \dd \iota_{\tilde{W}} )\,, \]
and
\[ \Ker({}^t \dd \phi)=\left(\HH^0(C,\omega_C(-D))+\langle s_R \rangle \right)\cap \HH^0(N,\omega_N) \,. \]

\end{proof}
We have the following:
\begin{theorem}\label{Thm: Characteristic Cycle of Xi}
Suppose $\ud=(g)$ or $\ud=(1,g-1)$. If $g$ is even, then 
\[ \mathrm{CC}(\IC_\Xi)=\Lambda_\Xi\,. \]
If $g$ is odd, then
\[ \mathrm{CC}(\IC_\Xi)=\Lambda_\Xi+\sum_{x\in \Xiasing } 2 \Lambda_x\,, \]
where $\Xiasing\coloneqq \phi(b^{-1}(\{D\in W\,|\, D\leq R\}))$ and $\Lambda_x=N^\vee_x P$ is the conormal variety to the point $x\in P$.
\end{theorem}
\begin{proof}
Let $\ud=(g)$ or $\ud=(1,g-1)$. By \ref{Cor: tilde W is smooth}, \ref{Prop: description of tNd} and \ref{Prop: Martens Thoerem for singular bielliptic curves} we know that $\phi^o:\tilde{W}^o\to \Xi^o$ is a small resolution of singularities. Thus by \cite[Prop. 5.4.4]{KashiwaraShapira1990} we have
\[ \mathrm{CC}(\IC_{\Xi^o}) \subseteq \phi_\pi({}^t \dd\phi^{-1} (N^\vee_{\tilde{W}^o} \tilde{W}^o)) \]
where ${}^t\dd \phi$ is the codifferential and $\phi_\pi$ is the projection
\[ T^\vee \tilde{W}^o \overset{{}^t \dd \phi}{\longleftarrow} \tilde{W}^o\times_{\Xi^o} T^\vee P \overset{\phi_\pi}{\longrightarrow} T^\vee P \,, \]
and $N^\vee_{\tilde{W}^o} \tilde{W}^o\subset T^\vee \tilde{W}^o$ is the zero section. Let $\Lambda={}^t \dd \phi^{-1}(N^\vee_{\tilde{W}^o}\tilde{W}^o)$. We define the following stratification of $W^o \coloneqq W\setminus \Wasing$
\begin{itemize}
    \item $W_k\subset W^o$ is the locus of non-singular divisors $D\in W$ which can be written as
    \[ D=\pi^\ast M+F\,, \]
    with $M\in E_k$ and $F$ $\pi$-simple. We have $\dim W_k=g-k-1$.
    \item $V_k\subset W^o$ is the locus of singular divisors $D\in W^o$ which can be written as
    \begin{align*} 
    D&=\pi^\ast M+F \,,
    \end{align*}
    with $M\in E_k$, and $F$ $\pi$-simple. We have $\dim V_k=g-k-2$.
\end{itemize}
For a locus $Z\subset W^o$, denote $\Lambda\restr{Z}$ the fiber of $\Lambda$ above $Z$. By Lemma \ref{Lem: Kernel of dd alpha in singular case} we have
\begin{enumerate}
    \item The fibers of $\Lambda\restr{W_0}\to W_0$ are of dimension $1$, Thus $\Lambda\restr{W_0}$ is of dimension $g$.
    \item Let $0<k\leq g-1$, and $D=\pi^\ast M+F\in W_k$. An element of $\HH^0(C,\omega_C(-D))$ vanishes at two conjugate points thus must be in the $(+)$-eigenspace of $\tau$, $\HH^0(C,\omega_C)^+=\HH^0(N,\omega_N)$. Thus
    \[ \HH^0(C,\omega_C(-D))=\HH^0(N,\omega_N(-D))\,, \]
    thus $\Ker({}^t \dd_D\phi)=\HH^0(N,\omega_N(-D))$ which is of dimension $k$ by Riemann-Roch. Thus $\dim \Lambda\restr{W_k}=g-1$.
    \item Let $0\leq k \leq g-2$ and $D\in V_k$. By \ref{Lem: Kernel of dd alpha in singular case} we have
    \[ \dim \Ker({}^t\dd_D \phi)\leq \hh^0(N,\omega_N(-\beta^\ast D))+1=k+1 \,, \]
    using Riemann-Roch. Thus $\dim \Lambda\restr{V_k}\leq g-1$.
\end{enumerate}
It follows that $\Lambda$, and thus $\phi_\pi(\Lambda)$ is irreducible of dimension $g$, which proves the theorem away from $\Xiasing$. Finally, the points in $\Xiasing$ are isolated quadratic singularities of maximal rank by \ref{Cor: points in Xiasing a isolated singularities of maximal rank}. For such a singularity, it is well-known that the characteristic cycle is irreducible if $g$ is even, and contains the conormal variety to the singular points with multiplicity $2$ if $g$ is odd.
\end{proof}

\printbibliography

@article{AltmanKleiman1980PicScheme,
title = {Compactifying the Picard scheme},
journal = {Advances in Mathematics},
volume = {35},
number = {1},
pages = {50-112},
year = {1980},
issn = {0001-8708},
doi = {https://doi.org/10.1016/0001-8708(80)90043-2},
url = {https://www.sciencedirect.com/science/article/pii/0001870880900432},
author = {Allen B Altman and Steven L Kleiman}
}

@misc{podelski2023GaussEgt,
      title={The Gauss map on bielliptic Prym varieties}, 
      author={Constantin Podelski},
      year={2023},
      eprint={2311.13521},
      archivePrefix={arXiv},
      primaryClass={math.AG}
}

@article{Shokurov1984PrymVarieties,
doi = {10.1070/IM1984v023n01ABEH001459},
url = {https://dx.doi.org/10.1070/IM1984v023n01ABEH001459},
year = {1984},
month = {2},
publisher = {},
volume = {23},
number = {1},
pages = {83},
author = {V V Shokurov},
title = {Prym Varieties: Theory and applications},
journal = {Mathematics of the USSR-Izvestiya}
}

@article{Naranjo1992BiellipticPryms,
url = {https://doi.org/10.1515/crll.1992.424.47},
title = {Prym varietes of bi-elliptic curves},
author = {Juan-Carlos Naranjo},
pages = {47--106},
volume = {1992},
number = {424},
journal = {Journal für die reine und angewandte Mathematik (Crelles Journal)},
doi = {doi:10.1515/crll.1992.424.47},
year = {1992},
lastchecked = {2023-11-18}
}

@misc{donagi1992fibers,
      title={The fibers of the Prym map}, 
      author={Ron Donagi},
      year={1992},
      eprint={alg-geom/9206008},
      archivePrefix={arXiv},
      primaryClass={alg-geom}
}

@book{Serre1988AlgGroupsClassFields,
  doi = {10.1007/978-1-4612-1035-1},
  url = {https://doi.org/10.1007/978-1-4612-1035-1},
  year = {1988},
  publisher = {Springer New York},
  author = {Jean-Pierre Serre},
  title = {Algebraic Groups and Class Fields}
}

@article{BresslerBrylinski97,
author = {Bressler, Paul and Brylinski, J.},
year = {1997},
month = {02},
pages = {},
title = {On the singularities of theta divisors on Jacobians},
volume = {7},
journal = {Journal of Algebraic Geometry}
}

@article{OdaSeshadri1979GenJac,
 ISSN = {00029947},
 url = {http://www.jstor.org/stable/1998186},
 author = {Tadao Oda and C. S. Seshadri},
 journal = {Transactions of the American Mathematical Society},
 pages = {1--90},
 publisher = {American Mathematical Society},
 title = {Compactifications of the Generalized Jacobian Variety},
 urldate = {2023-04-13},
 volume = {253},
 year = {1979}
}

@book{KashiwaraShapira1990,
  doi = {10.1007/978-3-662-02661-8},
  url = {https://doi.org/10.1007/978-3-662-02661-8},
  year = {1990},
  publisher = {Springer Berlin Heidelberg},
  author = {Masaki Kashiwara and Pierre Schapira},
  title = {Sheaves on Manifolds}
}

@book{Fulton1998,
  doi = {10.1007/978-1-4612-1700-8},
  url = {https://doi.org/10.1007/978-1-4612-1700-8},
  year = {1998},
  publisher = {Springer New York},
  author = {William Fulton},
  title = {Intersection Theory}
}

@article{Beauville1977,
author = {Beauville, Arnaud},
journal = {Inventiones mathematicae},
pages = {149-196},
title = {Prym Varieties and the Schottky Problem.},
url = {http://eudml.org/doc/142490},
volume = {41},
year = {1977},
}

@article{Kraemer2021MicrolocalGauss2,
   title={Characteristic cycles and the microlocal geometry of the Gauss map, II},
   volume={2021},
   ISSN={1435-5345},
   url={http://dx.doi.org/10.1515/crelle-2020-0048},
   DOI={10.1515/crelle-2020-0048},
   number={774},
   journal={Journal für die reine und angewandte Mathematik (Crelles Journal)},
   publisher={Walter de Gruyter GmbH},
   author={Krämer, Thomas},
   year={2021},
   month=jan, pages={53–92} }

@article{Debarre1988,
  doi = {10.1215/s0012-7094-88-05711-0},
  url = {https://doi.org/10.1215/s0012-7094-88-05711-0},
  year = {1988},
  month = {8},
  publisher = {Duke University Press},
  volume = {57},
  number = {1},
  pages = {221--273},
  author = {Olivier Debarre},
  title = {Sur les varietes abeliennes dont le diviseur theta est
		 singulier en codimension {3}},
  journal = {Duke Mathematical Journal}
}

@book{arbarello,
  title={Geometry of Algebraic Curves I},
  author={Arbarello, E. and Cornalba, M. and Griffiths, P. and Harris, J.D.},
  number={Bd. 1},
  isbn={9781475753233},
  series={Grundlehren der mathematischen Wissenschaften},
  url={https://books.google.de/books?id=LanxBwAAQBAJ},
  year={1985},
  publisher={Springer New York}
}

@book{arbarello2,
  doi = {10.1007/978-3-540-69392-5},
  url = {https://doi.org/10.1007/978-3-540-69392-5},
  year  = {2011},
  publisher = {Springer Berlin Heidelberg},
  author = {Arbarello, E. and Cornalba, M. and Griffiths, P.},
  title = {Geometry of Algebraic Curves II}
}

\end{document}